\documentclass[12pt,a4paper]{article}

\makeatletter
\def\thickhrulefill{\leavevmode \leaders \hrule height 1ex \hfill \kern \z@}
\def\maketitle{
  \vspace*{0\p@}%
  {\parindent \z@ \centering \reset@font%
        \thickhrulefill \quad \scshape \@author \ -\ \@date \quad \thickhrulefill%
        \par\nobreak%
        \vspace*{10\p@}%
        \interlinepenalty\@M
        \hrule
        \vspace*{10\p@}%
        \Large \bfseries \@title \par\nobreak
        \par
        \vspace*{10\p@}%
        \hrule
    \vskip 30\p@
}} \makeatother

\makeatletter
\def\blfootnote{\xdef\@thefnmark{}\@footnotetext}
\makeatother

\usepackage{pifont,fancybox,enumerate,multirow,bigstrut,amsmath,amssymb,amsfonts,graphicx,xypic}
\usepackage{amsthm,url,hyperref,titlesec,xcolor,wasysym}
\usepackage[english]{babel}
\usepackage[all]{xy}

\numberwithin{equation}{section}

\newtheoremstyle{rema}%
{10pt}%
{10pt}%
{}%
{}%
{\itshape}%
{\ --}%
{0.5em}%
{}

\newtheorem{lemma}{Lemma}[section]
\newtheorem{prop}[lemma]{Proposition}
\newtheorem{thm}[lemma]{Theorem}
\newtheorem{corr}[lemma]{Corollary}
\newtheorem{defin}[lemma]{Definition}

\theoremstyle{rema}
\newtheorem{rem}[lemma]{Remark}

\newcommand{\Ext}{\ensuremath{\text{Ext}}}

\newcommand{\Gal}{\ensuremath{\text{Gal}}}
\newcommand{\Hom}{\ensuremath{\text{Hom}}}

\newcommand{\Inv}{\ensuremath{\text{Inv}}}
\newcommand{\id}{\ensuremath{\text{id}}}

\newcommand{\SK}{\ensuremath{SK}}
\newcommand{\SKb}{\ensuremath{\textbf{SK}}}
\newcommand{\car}{\ensuremath{\text{char}}}
\newcommand{\ind}{\ensuremath{\text{ind}}}
\newcommand{\per}{\ensuremath{\text{per}}}

\newcommand{\Br}{\ensuremath{\text{Br}}}
\newcommand{\SL}{\ensuremath{\text{SL}}}
\newcommand{\GL}{\ensuremath{\text{GL}}}
\newcommand{\SLb}{\ensuremath{\textbf{SL}}}
\newcommand{\spec}{\ensuremath{\text{Spec}}}
\newcommand{\Symd}{\ensuremath{\text{Symd}}}

\newcommand{\kbar}{\ensuremath{\overline{k}}}

\newcommand{\GB}{\ensuremath{\textbf{G}}}
\newcommand{\HB}{\ensuremath{\textbf{H}}}
\newcommand{\PGSpB}{\ensuremath{\textbf{PGSp}}}
\newcommand{\PGLb}{\ensuremath{\textbf{PGL}}}
\newcommand{\AutB}{\ensuremath{\textbf{Aut}}}
\newcommand{\Aut}{\ensuremath{\text{Aut}}}
\newcommand{\Abar}{\ensuremath{\overline{A}}}

\newcommand{\Nrd}{\ensuremath{\text{Nrd}}}
\newcommand{\Nrp}{\ensuremath{\text{Nrp}}}
\newcommand{\Trd}{\ensuremath{\text{Trd}}}
\newcommand{\Trp}{\ensuremath{\text{Trp}}}
\newcommand{\Prd}{\ensuremath{\text{Prd}}}
\newcommand{\Prp}{\ensuremath{\text{Prp}}}

\newcommand{\kfields}{\ensuremath{k\text{-}\mathfrak{fields}}}

\newcommand{\Groups}{\ensuremath{\mathfrak{Groups}}}
\newcommand{\Sets}{\ensuremath{\mathfrak{Sets}}}

\newcommand{\SB}{\ensuremath{\text{SB}}}

\newcommand{\cd}{\ensuremath{\text{cd}}}
\newcommand{\rhokahn}{\ensuremath{\rho_{\text{Kahn}}}}
\newcommand{\rhokahna}[1]{\ensuremath{\rho_{\text{Kahn},#1}}}
\newcommand{\rhokmrt}{\ensuremath{\rho_{\text{BI}}}}
\newcommand{\rhokmrta}[1]{\ensuremath{\rho_{\text{BI},#1}}}
\newcommand{\rhosus}{\ensuremath{\rho_{\text{S06}}}}
\newcommand{\rhosusa}[1]{\ensuremath{\rho_{\text{S06},#1}}}
\newcommand{\rhosusoud}{\ensuremath{\rho_{\text{S91}}}}
\newcommand{\rhosusouda}[1]{\ensuremath{\rho_{\text{S91},#1}}}

\newcommand{\rhokahnga}[1]{\ensuremath{\tilde{\rho}_{\text{Kahn},#1}}}
\newcommand{\rhosusg}{\ensuremath{\tilde{\rho}_{\text{S06}}}}
\newcommand{\rhosusga}[1]{\ensuremath{\tilde{\rho}_{\text{S06},#1}}}

\newcommand{\rhorost}{\ensuremath{\rho_{\text{Rost}}}}
\newcommand{\rhorosta}[1]{\ensuremath{\rho_{\text{Rost},#1}}}

\newcommand{\pform}[1]{{\langle\!\langle{#1}\rangle\!\rangle}}

\newcommand{\Acal}{\ensuremath{\mathcal{A }}}
\newcommand{\Bcal}{\ensuremath{\mathcal{B }}}

\newcommand{\Hcal}{\ensuremath{\mathcal{H }}}

\newcommand{\Ocal}{\ensuremath{\mathcal{O }}}

\newcommand{\Qcal}{\ensuremath{\mathcal{Q }}}

\newcommand{\Tcal}{\ensuremath{\mathcal{T }}}

\newcommand{\Xcal}{\ensuremath{\mathcal{X }}}

\newcommand{\Fb}{\ensuremath{\mathbb{F }}}
\newcommand{\Gb}{\ensuremath{\mathbb{G }}}

\newcommand{\Qb}{\ensuremath{\mathbb{Q }}}

\newcommand{\Zb}{\ensuremath{\mathbb{Z }}}

\titleformat{\subsubsection}[runin]
  {\itshape}
  {\thesubsubsection}
  {5pt }
  {}[\ --]

\renewcommand{\thesubsubsection}{(\alph{subsubsection})}

\title{Comparing invariants of $\SKb_1$}
\author{Tim Wouters}
\date{March 18, 2010}

\begin{document}

\maketitle

\blfootnote{\ \\
K.U.Leuven -- Departement Wiskunde, Celestijnenlaan 200B bus 2400, B-3001 Leuven, Belgium --
tim@wouters.in
\\ 
\textit{2010 Mathematics Subject Classification:} 19B99 (12G05, 16K50, 17B20)\\
\textit{Keywords:} Reduced Whitehead Group -- Suslin's conjecture -- Cohomological invariants -- Symbol algebras}

\begin{abstract}
In this text, we compare several invariants of the reduced Whitehead group $\SKb_1$ of a central simple
algebra.  

For biquaternion algebras, we compare a generalised invariant of Suslin as constructed
by the author in \cite{wsuspos} to an invariant introduced by Knus-Merkurjev-Rost-Tignol \cite{kmrt}.  
Using explicit computations, we prove
these invariants are essentially the same.

We also prove the non-triviality of an invariant introduced by
Kahn \cite{kahnsk12}.   
To obtain this result, we compare Kahn's invariant to an invariant introduced by Suslin in 1991 \cite{suslinconj}  which is non-trivial
for Platonov's examples of non-trivial $\SK_1$ \cite{sk1niettriveng}.  
We also give a formula for the value on the centre of the tensor product of two symbol algebras which generalises a formula
of Merkurjev for biquaternion algebras \cite{merkrostinv}.
\end{abstract}

\section{Introduction} \label{sec:intro}

Let $k$ be a field and $A$ a central simple $k$-algebra.  The triviality of the \textit{reduced Whitehead
group} $\SKb_1(A)$ (which is isomorphic to $\SLb_1(A)/[A^\times,A^\times]$) is a long studied question.
Tannaka and Artin posed the question in the 1930's \cite{nakmat,wang}.  For more than 30 years, one tried
to prove the triviality of $\SKb_1(A)$ in full generality.  
In 1976, Platonov gave a counterexample using discrete valuation rings \cite[Thm. 5.19]{sk1niettriveng}.  Wang, however, did prove the triviality of 
$\SKb_1(A)$ if $\ind_k(A)$ is square-free \cite{wang}.   This inspired Suslin to conjecture that
this can be the only case of triviality \cite{suslinconj}.  This would give a sufficient answer to the question of Tannaka and Artin.
Merkurjev proved it is true when $4\,|\, \ind_k(A)$  \cite{mersuslinbiquat}; and Rehmann-Tikhonov-Yanchevski\u{\i} proved
it is sufficient to prove the conjecture for the tensor product of two symbol algebras \cite[Thm. 0.19]{rehmanea}.

In order to study his conjecture, Suslin conjectured in 1991 the existence of a cohomological invariant of $\SKb_1(A)$  with values in Galois cohomology ($n=\ind_k(A)\in k^\times$):
\begin{equation} 
\rho : \SKb_1 (A)(k) \to H^4(k, \mu_n^{\otimes 3} ) / (H^2(k,\mu_n^{\otimes 2}) \cup [A]), \label{eq:rhosusconj}
\end{equation}
where $[A]$ stands for the class of $A$ in $\phantom{ }_n\Br(k)\cong H^2(k,\mu_n)$ \cite[Conj. 11.6]{suslinconj}.  
There are various definitions of invariants of this flavour.  In 1991, Suslin defined twice his conjectured invariant
 (ibid., \S 2).  For biquaternion algebras Rost gave a closely related invariant \cite[Thm. 4]{merkrostinv}, and in 2006 Suslin
defined his conjectured invariant in full generality \cite[\S 6]{suslin}.  Kahn even generalised this invariant to a range of new invariants \cite[Cor. 8.4 \& Def. 11.3]{kahnsk12}.  

The restriction to central simple algebras with $n=\ind_k(A)\in k^\times$ is a natural one, since
otherwise the cohomology groups can be trivial and $\phantom{ }_n\Br(k)$ does not have to be isomorphic to $H^2(k,\mu_n)$.  Using Kato's
cohomology of \textit{logarithmic differentials} \cite{katogalcoh}, the author generalised any of the  aforementioned invariants to all central simple
algebras using a lift from positive characteristic to characteristic 0 \cite{wsuspos}.  We  recall the definitions
of the invariants in more detail in Section \ref{sec:recall}.  It is generally assumed that all defined invariants are essentially the same, but very few results exist on this subject.  In this paper, we compare some of them.

For biquaternion algebras,
Knus-Merkurjev-Rost-Tignol constructed a cohomological invariant of $\SKb_1(A)$ without the condition on the index \cite[\S 17]{kmrt}.  They use Witt groups, Witt rings, and a involution on the biquaternion algebra to define it.
If $\car(k)\neq 2$, they prove the invariant is essentially the same as Suslin's invariant for biquaternions.
Using the construction of the generalisation of Suslin's invariant, we prove that for base fields
of characteristic 2 their invariant essentially equals Suslin's generalised invariant (Section \ref{sec:biquat}).

In Section \ref{sec:compkahn}, we compare a new invariant of Kahn with all of the other existing invariants (or more correctly, we compare the invariants to Kahn's invariant).  
This allows us to prove the non-triviality of Kahn's invariant for Platonov's examples of non-trivial
$\SK_1$.  We also prove a formula for the value on the centre of the tensor product of two symbol algebras
under Kahn's invariant which generalises a formula  of Merkurjev for biquaternion algebras (\cite[Ex. p.70]{merkrostinv} -- see also \cite[Ex. 17.23]{kmrt}).  

\subsubsection*{Notations}
\addcontentsline{toc}{subsubsection}{Notations} 
Let us fix the following notations throughout this text.
\begin{itemize}
\item  If $k$ is a field, then $k_s$ denotes a separable closure and $\Gamma_k=\Gal(k_s/k)$ its absolute Galois
group.  Furthermore, denote $\Gb_m=\spec (\Zb[t,t^{-1}])$.
\item A prime factorisation $p_1^{e_1}\cdot \ldots\cdot
p_r^{e_r}$ of a (positive) integer $m$ is always supposed to primitive (i.e. $m=p_1^{e_1}\cdot \ldots\cdot
p_r^{e_r}$, with $p_i$ primes, $e_i\geq 1$ integers for $1\leq i \leq n$ and $p_i\neq p_j$ for any $1\leq i<j\leq r$)
\item We use standard notations for the following categories: the category $\Sets$ of set, the category $\kfields$ of field extensions of a field $k$, 
the category $\Groups$ of groups, and the category $\mathfrak{Ab}$ of abelian groups.
\item If $m>0$ is an integer (prime to $\car(k)$), then $\mu_m$ denotes the $\Gamma_k$-module of consisting of $m$-th roots of unity of $k_s$. If we want to stress the field in use, we write $\mu_m(k)$ (so that this can be viewed as the $k$-rational
points of the appropriate sheaf).
\item  The appearing cohomology groups are Galois cohomology groups (unless mentioned otherwise).
\item $\phantom{ }_m \Br(k)$ is the $m$-th torsion part of the Brauer group of $k$ ($m>0$ an integer).  If $K$ is a field
extension of $k$, we denote by $\Br(K/k)$ the kernel of the base extension morphism $\Br(k)\to \Br(K)$.
\item If $F$ is a discrete valuation field (with valuation $v$), then the valuation ring is denoted by $\Ocal_v$ and the residue field
by $\kappa(v)$.
If $x \in \Ocal_v$, we denote by $\bar{x}$ its class in $\kappa(v)$.  
We  also use this notation for other objects for which we can define (canonical) residues.
A discrete valuation is supposed to be non-trivial (of rank 1).  By $F_{\text{nr}}$ we denote the
maximal unramified extension of $F$.
\item If $A$ is a central simple $k$-algebra and if $F$ is a field extension of $k$, then $A_F=A\otimes_k F$ 
is the central simple $F$-algebra obtained from $A$ by base extension  to $F$.  More generally, for a ring  $R$, a commutative $R$-algebra $S$, and an Azumaya $R$-algebra $A$, we denote  $A_S=A\otimes_R S$, the Azumaya $S$-algebra obtained from $A$ by 
base extension to $S$.  By $[A]$ we denote the Brauer class of a central simple algebra/Azumaya algebra $A$.
\item\label{it:red} For any central simple algebra $A$ and $a\in A$, the reduced norm of $a$ is denoted as $\Nrd_{A/k}(a)$.  In the same way,
$\Trd_{A/k}(a)$ is the reduced trace and $\Prd_{A,a/k}(X)$ is the reduced characteristic polynomial.  
If \[ \Prd_{A,a/k}(X)=X^n-s_1(a)X^{n-1}+\ldots +(-1)^n s_n(a),\] then we know $\Nrd_{A/k}(a)=s_n(a)$ and $\Trd_{A/k}(a)=s_1(a)$.
\item For a central simple $k$-algebra, we denote by $\SLb_1(A)$ the usual linear algebraic group scheme.  If $F$ is a field
extension of $k$, the $F$-rational points of $\SLb_1(A)$ are given by
$ \SLb_1(A)(F) = \{ x\in A_F\, |\, \Nrd_{A/F}(a)=1\}$.
Furthermore, by $\SKb_1(A)$ we denote the group functor 
\[ \kfields \to \Groups : F \mapsto \SKb_1(A)(F)=SK_1(A_F)\cong \SLb_1(A)(F)/[A^\times,A^\times], \]
called the \textit{reduced Whitehead group of $A$}.
\end{itemize}

\subsubsection*{Acknowledgements}  The author thanks Jean-Louis Colliot-Th\'el\`ene, Philippe Gille, Bruno Kahn, Marc Levine, and Jean-Pierre Tignol for interesting comments
and discussions
leading to this article.  The author expresses his gratitude to the \'Ecole Normale Sup\'erieure (Paris), K.U.Leuven, 
and Research Foundation - Flanders (G.0318.06) for financial support.

\section{Existing invariants and results} \label{sec:recall}

In this section, we  recall the various invariants of $\SKb_1$ introduced by several authors.  All differ a little bit
on the value groups.  It takes some time to introduce all of them in quite a rigorous way.  The author excuses 
to the reader who is aware of all the definitions and results and hopes for his tolerance (and endurance).  
He believes it makes the text
more accessible for the non-expert.
Experts can eventually skip this section at first and come back to it if necessary to find e.g. a particular definition.  
Before recalling the invariants, we recall Merkurjev's viewpoint on invariants and Platonov's examples of non-trivial $\SK_1$ as both are used later on.

\subsection{Invariants \`a la Merkurjev} \label{sec:merkurjev}

For two group functors $\GB,\HB:\kfields \to \Groups$, an invariant of $\GB$ in $\HB$ is a 
natural transformation of functors of $\GB$ into $\HB$. Typically $\HB$ equals
the degree $j$ part $M_j$ of a cycle module $M$ (\`a la Rost \cite{rostmodcyc}),
such an invariant is called an \textit{invariant of $\GB$ in $M$ of degree $j$}.
It is clear that all invariants of $\GB$ in $\HB$ form a group (abelian if $\HB$ has images in $\mathfrak{Ab}$).  In 
case of  degree $j$ invariants of $\GB$ in a cycle module $M$, we denote this group by
$\Inv^j(\GB,M)$.   We can define the same terminology  if $M$ is any functor of graded groups.

\subsubsection{Cycle modules} \label{sec:defcycmod}
A cycle module $M$ having a field $k$ as base is a formal object having the shared properties of certain
Galois cohomology groups, Milnor's $K$-groups, \ldots\ It associates with any field extension $F$ of $k$ a graded
abelian group $(M_j)_{j\geq 0}$ endowed with four data (functoriality, reciprocity, $K$-theory module structure, and residues -- D1-D4 in ibid., Def. 1.1) satisfying some homological and geometrical
rules (R1a-R3e, FD, and C - ibid., Def. 1.1  \& 2.1).
For a field $k$, a central simple $k$-algebra $A$ of $n=\ind_k(A)\in k^\times$,
 and an integer $m\in k^\times$, we  use the following cycle modules (for any integer $r$): 
\begin{eqnarray}
\Hcal^\ast_m: \kfields \mapsto \mathfrak{Ab} &:& F \mapsto (H^j_m(F))_{j>0} \label{eq:hm} \\
\Hcal^\ast_{n,\Acal^{\otimes r}}: \kfields \mapsto \mathfrak{Ab} &:& F \mapsto (H^j_{n,A^{\otimes r}}(F))_{j> 1}, \label{eq:hna}
\end{eqnarray} 
with 
\begin{equation} \label{eq:cyclemod}
H^j_m(F)=H^j(F,\mu_m^{\otimes (j-1)}) \quad \text{ and } \quad
 H^j_{n,A^{\otimes r}}(F)=H^j_n(F)/ \bigl(H^{j-2}(F,\mu_n^{\otimes (j-2)}) \cup r[A]\bigr).  
\end{equation}
Remark that if $r\equiv 0 \mod \per_k(A)$, then $H^j_{n,A^{\otimes r}}(F)=H^j_{n}(F)$.  So the second cycle module
is actually a generalisation of the first one.

\subsubsection{Gersten complex}
Given a $k$-variety $X$ and a cycle module $M$ with base $k$, we have a Gersten
complex (ibid., \S 3.3) (for integers $i,j\geq 0$):
\[ \ldots \overset{\partial^{i-2}}{\to} 
\bigoplus_{x\in X^{(i-1)}} M_{j-i+1}(k(x)) \overset{\partial^{i-1}}{\to}
\bigoplus_{x\in X^{(i)}} M_{j-i}(k(x)) \overset{\partial^{i}}{\to}
\bigoplus_{x\in X^{(i+1)}} M_{j-i-1}(k(x)) \overset{\partial^{i+1}}{\to}
\ldots,
\]
induced by the residues of the cycle module.  Here, $X^{(i)}$ is the set
of points of $X$ of codimension $i$, $k(x)$ is the function field of a point of
codimension $i$, and any appearance of negative degree of the cycle module
is to be interpreted as the trivial group.
The homology of this complex on spot $i$ is denoted $A^i(X,M_j)$.

\subsubsection{Merkurjev's link}
Let $\GB$ be an algebraic $k$-group which we view  as a group functor
associating to a field extension $F$ of $k$, the group $\GB(F)$ of $F$-rational points  of $\GB$.
If $M$ is of bounded exponent, then Merkurjev gives an isomorphism 
\begin{equation} \label{eq:isomerk}
\Inv^j(\GB,M) \overset{\sim}{\to} A^0(\GB,M_j)_{\text{mult}} \subset A^0(\GB,M_j): \rho \mapsto \rho_{K}(\xi),
\end{equation}
where $K=k(\GB)$ is the function field of $\GB$ and $\xi\in \GB(K)$ the
generic point of $\GB$ \cite[Lem. 2.1 \& Thm. 2.3]{invalggroup}.  
The image $A^0(\GB,M_j)_{\text{mult}}$ consists of the \textit{multiplicative elements} of $A^0(\GB,M_j)$.  
These are those elements $x$ such that
$p_1^\ast(x) + p_2^\ast(x) = m^\ast(x)$ for $p_1^\ast,p_2^\ast,m^\ast$ induced by
the projection $p_1,p_2:\GB \times \GB \to \GB$ and multiplication $m:\GB \times \GB \to \GB$.

\subsection{Platonov's examples} \label{sec:platonov}

Among the examples of non-trivial $\SK_1$ of Platonov, we  concentrate on the tensor
product of two cyclic algebras.

\subsubsection{Cyclic algebras} \label{sec:cyclicalgs}
Let $k$ be a field and $K$ a cyclic field extension of degree $n$.  Take furthermore a generator $\sigma\in \Gal(K/k)\cong \Zb/n$.  Then for $b\in k^\times$, we denote by $(K/k,\sigma,b)$ the so-called \textit{cyclic $k$-algebra} 
generated by $K$ and a variable $x$ satisfying $x^n=b$ and $xc=\sigma(c)x$ for any $c\in K$.  
Then clearly $\deg_k(K/k,\sigma,b)=n$ and we can also write this cyclic algebra as
$\oplus_{i=0}^{n-1} K x^i$ with multiplication defined as above \cite[\S 7, Def. 4]{draxl}.   Furthermore,
$K$ is a splitting field of $(K/k,\sigma,b)$ (see \cite[\S 2.5]{gilleszam}).  

If $k$ contains an $n$-th primitive root of unity and
if $K=k(\sqrt[n]{a})$ for $a\in k^\times$, then $(K/k,\sigma,b)\cong (a,b)_n$ as $k$-algebras (if $\sigma$ is well chosen).  Here $(a,b)_n$ is the usual \textit{symbol $k$-algebra}
generated over $k$ by variables $x$ and $y$ satisfying $x^n=b,y^n=b$, and $xy=\xi_nyx$ for a well chosen primitive 
$n$-th root of unity $\xi_n\in k$.  In case $n=p=\car(k)$ and if $K$ is the cyclic Galois extension defined by $x^p-x-a$,
then $(K/k,\sigma,b)\cong [a,b)_p$ as $k$-algebras (for a well chosen $\sigma$).  Here $[a,b)_p$ is the usual \textit{$p$-algebra}: generated as $k$-algebra
by variables $x$ and $y$ satisfying $x^p-x=a,y^p=b$, and $xy=y(x+1)$ (loc. cit.).

If $n=2$, a symbol algebra or $p$-algebra is more commonly called a \textit{quaternion algebra}.
The product of two quaternion algebras is a \textit{biquaternion algebra}; it is a central simple algebra of
degree 4 and period 1 or 2.  It is know that biquaternion algebras are in fact the only central simple algebras of
degree 4 and period 1 or 2  \cite[p. 369]{albertbiquat}. 

\subsubsection{Non-trivial $\SK_1$} \label{sec:nontrivsk1}

Let $k$ be a local field (e.g. $\Qb_p$ or $\Fb_p((x))$) and let $K_1,K_2$ be two cyclic extensions of degree $n$ over
$k$ which are linearly disjoint.  Let $\sigma_1$ (resp. $\sigma_2$) be a generator of $\Gal(K_1/k)$ (resp. $\Gal(K_2/k)$).  
Now let $F=k((t_1))((t_2))$, $F_1=K_1((t_1))((t_2))$, and $F_2=K_2((t_1))((t_2))$.
Then Platonov proves that \[ A=(F_1/F,\sigma_1,t_1)\otimes (F_2/F,\sigma_2,t_2) \] is a division $F$-algebra
and furthermore $\SK_1(A)\cong \Br(K/k)/(\Br(K_1/k)\Br(K_2/k)) \cong \Zb/n$ for $K=K_1\otimes K_2$
\cite[Thms. 4.7 \& 5.9]{sk1niettriveng}.

\subsubsection{Galois cohomology of $\Qb_p((t_1))((t_2))$} \label{sec:galcoh}

To study the invariants later on, we  encounter the fourth Galois cohomology groups $H^4_{m}(k)$ for 
$k=\Qb_p((t_1))((t_2))$.  These can be calculated
using a splitting for a complete discrete valuation field $K$ with residue field $\kappa(v)$ and with $m\in \kappa(v)^\times$ (hence also $m\in K^\times$)  \cite[7.11]{cohinv}:
\begin{equation}
\label{eq:splitting} 
H^{i+1}_m(K) \cong H^{i+1}_m(\kappa(v)) \oplus H^{i}_m(\kappa(v)). 
\end{equation}
Using the fact that $\cd(\Qb_p)=2$ and $\Br(\Qb_p)=\Qb/\Zb$ \cite[Ch. II, \S 5.1 \& Prop. 15]{serregalcoh}, we find $H^{4}_m(k)\cong \Zb/m$ by applying the splitting to the valuations defined by $t_1$
and $t_2$.  

\subsection{Suslin's invariants} \label{sec:susinvariants}

We  recall the invariants of Suslin and an invariant for biquaternion algebras introduced by Rost.
Let us first give the motivation why these invariants can help to explain Platonov's counterexamples.

\subsubsection{Suslin 1991}   \label{sec:suslin91}
By constructing his invariant $\rho_A \in \Inv^4(\SKb_1(A), \Hcal^\ast_{m,\Acal})$ (for $m=\ind_k(A)\in k^\times$),
Suslin hoped 
to be able to complete the following diagram (for $A$ as in \S \ref{sec:platonov} \ref{sec:nontrivsk1}): 
\begin{equation} \label{diag:motiv}
\xymatrix{ 
\SK_1(A) \ar[rr]^{\cong\qquad \qquad \quad} \ar[d]_{\rho_{A,F}} & & \Br(K/k)/(\Br(K_1/k)\Br(K_2/k)) \ar[d] \\
H^4_{n^2,A}(F) \ar[rr]^{\partial^3_{t_1}\circ \partial^4_{t_2}\qquad \qquad \quad}&  & H^2_{n^2}(k)/\partial^3_{t_1}\circ \partial^4_{t_2}(H^2(k,\mu_{n^2}^{\otimes 2})\cup [A])
}
\end{equation}
The maps $\partial^3_{t_1},\partial^4_{t_3}$ are residues induced by the discrete valuation associated
with $t_1$ and $t_2$, i.e. the projection maps of degree $-1$ in \eqref{eq:splitting}.
At the time he conjectured the existence of such an invariant, he could not yet give a definition.  
He was however able to define an invariant
$\rhosusouda{A} \in \Inv^4(\SKb_1(A), \Hcal^\ast_{m,\Acal^{\otimes 2}})$ which 
he proves to be non-trivial for Platonov's examples of non-trivial $\SK_1$.

\subsubsection{Biquaternion algebras} \label{sec:invbiquat}

In the case of biquaternion algebras, Rost was able to define a related invariant of $\SKb_1(A)$.
Suppose $A$ is a biquaternion algebra over a field $k$ of $\car(k)\neq 2$.  Then \textit{Rost's invariant} 
$\rhorosta{A}$ is an invariant sitting in $\Inv^4(\SKb_1(A),\Hcal_{2}^\ast)$ \cite[Thm. 4]{merkrostinv}.  Moreover,
it fits into an exact sequence:
\[
 0 \to  \SKb_1(A)(k) \to H^4(k,\Zb/2\Zb) \to  H^4(k(Y),\Zb/2\Zb),
\]
where $Y$ an Albert form of $A$.  
This invariant was generalised in \cite[\S 17]{kmrt} to biquaternion algebras in any characteristic
using Witt groups and Witt rings.  We  come back to this generalised invariant in Section \ref{sec:biquat} 
as its definition requires a lot of  terminology related to involutions.

\subsubsection{Suslin 2006} \label{sec:suslin06}

Using Voevodsky's motivic \'etale cohomology, Suslin was able
to define his conjectured invariant in 2006 \cite[\S 3]{suslin}.  
We   denote this invariant by $\rhosusa{A}$.  
It is however not clear whether \eqref{diag:motiv} commutes for this invariant.
It is  clear that this (and also the other invariants)
become trivial after base extension to the function field of $X=\SB(A)$ (it is a splitting field of $A$).  
Suslin hence proves his invariant is essentially the same as 
Rost's invariant $\rhorosta{A}$ for a biquaternion algebra over a field $k$ of $\car(k)\neq 2$.  He does this by proving that
\begin{equation}  \label{eq:modcompb}
 \xymatrix{
 \SKb_1(A)(k) \ar[d]_{=}\ar[r]^{\rhosus\qquad \quad \ \ } & \ker\bigl[H^4_{4,A}(k) \to H^4_{4,A}(k(X))\bigr] \ar[d]^{r_A} \\
 \SKb_1(A)(k) \ar[r]_{\rhorost \qquad \quad}& \ker\bigl[H^4_2(k) \to H^4_2(k(Y))\bigr],
}
\end{equation}
is a commutative diagram, where is $r_A$ is the morphism induced on Galois cohomology by the map 
$\mu_4^{\otimes 3}\to \mu_2:a\mapsto a^2$ and where $X$ and $Y$ are as above.
Hence $\rhosus$ is injective for biquaternion algebras and  
\[
\SKb_1(A)(k) \cong \ker\bigl[H^4_{4,A}(k) \to H^4_{4,A}(k(X))\bigr].
\]

\subsection{Kahn's invariants} \label{sec:defkahn}

Let $k$ be a field and $A$ a central simple algebra with $n=\ind_k(A)\in k^\times$.  We  recall the inspiring results 
on invariants of $\SKb_1(A)$ as obtained by Kahn in \cite{kahnsk12}.

\subsubsection{Cyclicity of invariant group} \label{sec:cyclic}
By calculations with motivic \'etale cohomology, Kahn shows $A^0(\SLb_1(A),\Hcal^4_n)_{\text{mult}}$ is finite cyclic \cite[Def. 11.3]{kahnsk12}.
So by Merkurjev's isomorphism \eqref{eq:isomerk}, $\Inv^4(\SLb_1(A),\Hcal^\ast_{n})$ is finite cyclic.  As the canonical projection $\SLb_1(A)\to \SKb_1(A)$ induces
an injective morphism
\begin{equation} \label{eq:injinv} 
\Inv^4(\SKb_1(A),\Hcal^\ast_n) \to \Inv^4(\SLb_1(A),\Hcal^\ast_n), 
\end{equation}
we also find $\Inv^4(\SKb_1(A),\Hcal^\ast_n)$ to be cyclic.  Using Kahn's calculations (loc. cit.), we can pick a canonical  generator that we  call
\textit{Kahn's invariant} $\rhokahna{A}$ of $\SKb_1(A)$.

\subsubsection{Bounds on invariant group} \label{sec:bound}

Kahn also argues the size of $\Inv^4(\SLb_1(A),\Hcal^{\ast}_{n})$ is bounded by 
$\ind(A)/l$ if $n=\ind_k(A)$ is the power of a prime $l$ (ibid., Lem. 12.1).  
Hence the same holds for $\Inv^4(\SKb_1(A),\Hcal^{\ast}_{n})$ by \eqref{eq:injinv}.
For general $n$, Kahn's bound is retrieved using Brauer's decomposition theorem  \cite[Ch. 4, Prop. 4.5.16]{gilleszam}.
For any integer  $n$   with prime factorisation $p_1^{e_1}\cdot \ldots\cdot
p_r^{e_r}$, we denote
by $\overline{n}$ the integer $p_1^{e_1-1}\cdot \ldots\cdot p_r^{e_r-1}$.

\begin{lemma} \label{lem:boundgen}
Let $k$ be a field and $A$ a central simple algebra of $\ind_k(A)=n\in k^\times$.
Then \[ \bigl|\Inv^4(\SKb_1(A),\Hcal^\ast_n)\bigr|\leq \overline{n}. \]  
\end{lemma}

\begin{proof}
Let $p_1^{e_1}\cdot \ldots\cdot p_r^{e_r}$ be a prime decomposition of $n$, then  Brauer's decomposition
theorem (loc.cit.) gives division $k$-algebras $D_1,\ldots,D_r$ of $\ind_k(D_i)=p_i^{e_i}$ such that
$A$ is Brauer-equivalent to $D_1\otimes \ldots \otimes D_r$.  This even gives rise to a decomposition
$\SKb_1(A) \cong \SKb_1(D_1) \oplus \ldots \oplus \SKb_1(D_r)$ (ibid., Ch. 4, Ex. 9).  Recall
also that $\SKb_1(D_i)$ has $p_i^{e_i}$-torsion \cite[\S 23, Lem. 3]{draxl}. 
Then the result follows immediately from the primary result of Kahn and the isomorphism
$ H^4_n(k)\cong H^4_{p_1^{e_1}}(k)\oplus \ldots \oplus H^4_{p_r^{e_r}}(k)$. 
\end{proof}

\begin{rem} \label{rem:kahnbiquatnontriv}
As Kahn mentions,  this bound is sharp  for biquaternion division algebras \cite[\S 12]{kahnsk12}.  This follows from \cite[Prop. 4.9 \& Thm. 5.4]{invalggroup}.  In particular,
$\rhokahn$ is not trivial for biquaternion division algebras.  In \S \ref{sec:compkahnmod} \ref{sec:nontrivkahn},
we generalise this result.
\end{rem}

\subsubsection{Generalisation of Suslin's invariant}

Apart from using Merkurjev's viewpoint to define a new invariant, Kahn also generalises $\rhosus$ to 
 invariants
\[ \rho_r \in \Inv^4(\SKb_1(A) ,\Hcal^{\ast}_{n,\Acal^{\otimes r}}) \]
with $n\in \ind_k(A)\in k^\times$ and $r=1,\ldots,\per_k(A)-1$.  Suslin's invariant $\rhosus$ is retrieved
setting $r=1$.  It is not clear whether $\rhosusoud$ equals $\rho_2$.  
As mentioned in \S \ref{sec:susinvariants} \ref{sec:suslin06}, $\rhosus$ has its image in $\ker\bigl[H^4_{n,A}(F)\to H^4_{n,A}(F(X))\bigr]$ for $F$ a field extension of $k$ and $X$ the Severi-Brauer variety of $A$ \cite[\S 3]{suslin}.   Kahn generalises this to $\rho_r$
replacing $X$ by the generalised Severi-Brauer variety $\SB(r,A)$ (ibid., \S 8.B).	

He also gives a bound on the torsion of these invariants inside $\Inv^4(\SKb_1(A),\Hcal^{\ast}_{n,\Acal})$ if $l=\per_k(A)$ is a prime.
Indeed from (ibid., Thm. 7.1(c) \& Cor. 12.10) it follows that
they have
\begin{itemize}
 \item $l$-torsion if $\ind_k(A)=\per_k(A)=l>2$, 
 \item $l^2$-torsion if $\per_k(A)>\ind_k(A)=l>2$, and
 \item 2-torsion if $\per_k(A)=2$.
\end{itemize}
For   
a central simple $k$-algebra $A$ with $n=\ind_k(A)\in k^\times$ and $\per_k(A)=n/\overline{n}$, there is a similar statement
using a Brauer decomposition.
Take a prime factorisation  $n=p_1^{e_1}\cdot \ldots\cdot p_r^{e_r}$ 
and let  $D_1\otimes \ldots \otimes D_r$ be a Brauer decomposition of $A$ as in the proof of Lemma \ref{lem:boundgen}.
Then put $m=p_1^{f_1}\cdot \ldots \cdot p_r^{f_r}$, where $f_i=1$ if $p_i=2$ or if $\ind_k(D_i)=\per_k(D_i)=p_i>2$,
and $f_i=2$ if $\ind_k(D_i)>\per_k(D_i)=p_i>2$.  Then it is clear that $\rho_r$ has
$m$-torsion.

\subsection{Generalising invariants} \label{sec:generalising}

In \cite{wsuspos}, the author introduced a way of generalising the invariants of $\SKb_1(A)$ to any central simple $k$-algebra $A$
(so also when $\ind_k(A)\not\in k^\times$).  This is done using a lift from 
a field of positive characteristic to a field of zero characteristic where the invariants are always defined.
In this subsection, let $k$ be a field of $\car(k)=p>0$.
We first explain Kato's cohomology of \textit{logarithmic differentials} which are used in (loc. cit.) to generalise  $\Hcal^\ast_{n}$ 
when $p\,|\,n$.  This allows us to perform lifts from positive characteristic to characteristic 0.

\subsubsection{Logarithmic differentials} \label{sec:logdif}
For any integer $l> 0$, the cohomology groups
$H^{q+1}_{p^l}(k)$ are defined as 
\[ (W_l(k)\otimes \underset{q \text{ times}}{\underbrace{k^\times \otimes \ldots \otimes k^\times}})/I,\]
where $W_l(k)$ are the Witt vectors of length $l$ on $k$ and $I$ is the ideal generated by
\begin{enumerate}[(i)] 
 \item $w\otimes b_1 \otimes \ldots \otimes b_{q}$, satisfying $b_i=b_j$ for $1\leq i < j \leq q$, \label{it:J1}
 \item $(0,\ldots,0,a,0,\ldots,0)\otimes a \otimes b_2 \otimes \ldots \otimes b_{q}$, \label{it:J2}
\item $(w^{(p)}-w) \otimes b_1 \otimes \ldots \otimes b_{q}$, \label{it:J3}
\end{enumerate}
with $w\in W_l(k)$, $b_1,\ldots, b_q\in k^\times$, and $w^{(p)}=(a_1^p,\ldots,a_l^p)$ if $w=(a_1,\ldots,a_l)$.

For $l=1$, we can view $H^{q+1}_p(k)$ as the cokernel of
\[ F-1:\Omega_{k}^q \to \Omega_{k}^q/d\Omega_{k}^{q-1},\   \text{ defined by } \ x \frac{dy_1}{y_1} \wedge \ldots \wedge 
 \frac{dy_q}{y_q} \mapsto (x^p-x) \frac{dy_1}{y_1} \wedge \ldots \wedge 
 \frac{dy_q}{y_q} \mod d\Omega_{k}^{q-1},
\]
hence the terminology \textit{``logarithmic differentials''}.  (For $l=0$, set $H^{q+1}_{p^l}(k)=0$.)

In general, for an integer $n=p^lm>0$ ($l,m\geq 0$ integers with $p\nmid m$), we define 
\begin{equation*} \label{eq:defcohgen}
H^{q+1}_n(k)=H^{q+1}_{p^l}(k) \oplus H^{q+1}_m(k).
\end{equation*}
This is a generalisation of Galois cohomology, since this theory fills in some gaps in Galois cohomology.  
It gives for example a description of the $p^l$-th torsion part of the Brauer group, compatible with the prime-to-$p$ part: $\phantom{ }_{p^l}\Br(k)\cong H^2_{p^l}(k)$.  
So for any integer $n>0$ we get $\phantom{ }_n\Br(k)\cong H^2_n(k)$.  We can also define $\Hcal^\ast_n$ in the same
way as in \eqref{eq:hm}.  It is however not a cycle module, but rather
a functor of graded groups.  To obtain a cycle module we have to tweak it a little bit.  For this paper we do not
need a cycle module, so we rather work with this functor of graded groups to ease the discussion (see \cite[\S 4.1 (d)]{wsuspos} for more details -- see also Remark \ref{rem:l} infra).

Using this isomorphism, together with a scalar multiplication by \textit{Milnor's $K$-groups} on $(H^{q+1}_{p^l}(k))_{q\geq 0}$,
we can generalise the definition of $\Hcal^\ast_{n,\Acal}$ for a central simple $k$-algebra $A$ with arbitrary index.
Recall that Milnor $K$-groups $K^M_r(k)$ (for an integer $r\geq 0$) are
defined as
\[ \underset{r \text{ times}}{\underbrace{k^\times \otimes \ldots \otimes k^\times}}/J, \]
where $J$ is the ideal generated by $x_1\otimes \ldots \otimes x_i$ with $x_i+x_j=1$ for some $1\leq i < j \leq r$.
Elements of $K^M_r(k)$ are called \textit{symbols} and the generators $x_1\otimes \ldots \otimes x_r$ are called \textit{pure symbols}, commonly denoted  $\{x_1,\ldots,x_r \}$.  The scalar multiplication of $K^M_r(k)$ on $(H^{q+1}_{p^l}(k))_{q\geq 0}$
is given by
\begin{eqnarray*}
 \cdot : K^M_r(k) \times H^{q+1}_{p^l}(k) & \to & H^{r+q+1}_{p^n}(k), \quad \text{ defined by } \\
(\{x_1,\ldots,x_r\},w\otimes b_1\otimes \ldots \otimes b_q ) & \mapsto & w \otimes x_1 \otimes \ldots \otimes x_r \otimes b_1 \otimes \ldots \otimes b_q.
\end{eqnarray*}
This  allows us to define a relative version.  Before doing so, we recall that also
the cup-product definition of \eqref{eq:cyclemod} can be generalised using $K$-theory.  Indeed, the 
isomorphism $k^\times/(k^\times)^m \cong H^1(k,\mu_m)$ for any $m\in k^\times$ gives  the
\textit{Galois symbol} by taking the cup-product:
\begin{equation} \label{eq:galoissymbol}
 h^r_{m,k}: K^M_r(k) \to H^r(k,\mu_m^{\otimes r}). 
\end{equation}
The Bloch-Kato conjecture (proved by Voevodsky-Rost-Weibel \cite{blochkato,voevodblken,rostblken,weibelblk})
even says it is surjective with kernel $mK^M_r(k)$. 
Hence we get a scalar multiplication of $K^M_r(k)$ on $(H^{q+1}_m(k))_{q\geq 0}$:
\[ \cdot: K^M_r(k) \times H^{q+1}_m(k) \to H^{r+q+1}_m(k) \quad \text{defined by} \quad 
( a ,b) \mapsto h^r_{m,k}(a) \cup b. \]
For arbitrary $n$, this defines in total a $K^M_r(k)$-module structure on $(H^{q}_n(k))_{q>0}$. If $A$ is a central
simple $k$-algebra of $\ind_k(A)=n$, we can then define for any field extension $F$ of $k$ and integers $q\geq 0$ and $r$:
\[ H^{q+1}_{n,A^{\otimes r}}(F)=H^{q+1}_n(F)/(K^M_{q-1}(F)\cdot r[A_F]). \]
By the remarks above, this is clearly a generalisation of the moderate case.  If $r\equiv 0 \mod \per_k(A)$, then clearly 
$H^{q+1}_{n,A^{\otimes r}}(F)= H^{q+1}_{n}(F)$ (cfr. \S \ref{sec:merkurjev} \ref{sec:defcycmod}).  In the same way as in \eqref{eq:hna}, we 
obtain a functor of graded groups $\Hcal^\ast_{n,\Acal^{\otimes r}}$.

\subsubsection{Lifts} \label{sec:lifts}
We  now consider $k$ to be the residue field of a complete discrete valuation ring $R$ with fraction field $K$ of
$\car(K)=0$.   The specialisation map 
$\Br(R)\to \Br(k): [A] \mapsto [A\otimes_R k] $ is bijective \cite[Cor. 6.2]{grothbrauer} and
$\Br(R)\to \Br(K): [A] \mapsto [A\otimes_R K]$ is injective \cite[Thm. 7.2]{ausgold}.  So we have an inclusion
$\Br(k)\to \Br(K)$; given a central simple algebra $A$ over $k$, we get a \textit{lifted Azumaya algebra} $B$ over $R$
and an associated central simple algebra $B_K$ over $K$.  Because of the definition, $\ind_k(A)=\ind_K(B_K)$ 
and by a theorem of Platonov we get $\SKb_1(A)(k)\cong \SKb_1(B_K)(K)$ \cite[Thm. 3.12]{sk1niettriveng} -- see also \cite[Cor. 3.3]{wsuspos}.

Furthermore, there exists an injection $H^{i+1}_{n}(k) \to H^{i+1}_{n}(K)$; on the prime-to-$p$ parts of $H^{i+1}_{n}(k)$ defined by
\eqref{eq:splitting}, for general $n$ see \cite[Proof of Prop. 2]{katogalcoh} and \cite[Prop. 6.8]{izboldhin} (see also
Remark \ref{rem:definj} for 2-primary $n$).  This injection
also continues to the relative cohomology groups; i.e. there exists an injection $H^{i+1}_{n,A^{\otimes r}}(k) \to H^{i+1}_{n, B_K^{\otimes r}}(K)$ for any integer $r$ and $A$ and $B$ as above \cite[Prop. 4.10]{wsuspos}.

This  allows us to define an invariant for any central simple $k$-algebra, using the existence in the
characteristic 0 case.  In order to stay functorial, we have to use \textit{$p$-rings}.  A $p$-ring is a complete discrete valuation ring $R$ with residue field $k$ of $\car(k)=p>0$ and
whose maximal ideal is generated by $p$.  For a reference see e.g. \cite[\S 23]{matsumura} where $p$-rings can also be not complete, we however
always suppose them to be complete. For a $p$-ring $R$, the fraction field $K$ is of characteristic 0.  
Moreover, Cohen proved that given a field $k$ of $\car(k)>0$, there always exists  a $p$-ring with $k$ as residue field \cite{cohen}.  

For sake of convenience, we also use the following terminology.
\begin{defin}
Suppose $\rho$ is an invariant of $\SKb_1$ which is defined for any central simple algebra $A$ with index $n$ not divisible by the
characteristic of its base field and which has 
values in the Galois cohomology group $\Hcal^4_{n,\Acal^{\otimes r}}$ for $r$ a fixed integer.  
Then we say $\rho$ is a \textit{moderate invariant of $\SKb_1$ with values in $\Hcal^4_{\otimes r}$}.
We denote by $\rho_A$ the invariant for a central simple algebra $A$. 
\end{defin}

In \cite[Thm.4.20]{wsuspos}, the author proves the following theorem.

\begin{thm}  \label{thm:lift}
Let $k$ be a field of $\car(k)=p>0$ and $A$ a central simple $k$-algebra of $n=\ind_k(A)$.  Take $R$ a $p$-ring with residue field $k$ and fraction field $K$.  Let $B$ be the lifted Azumaya $R$-algebra of $A$ and let
$\rho\in \Inv^4(\SKb_1(B_K),\Hcal_{n,\Bcal^{\otimes r}_K}^\ast)$ (for $r$ any integer).  There exists a unique invariant $\tilde{\rho}\in \Inv^4(\SKb_1(A),\Hcal_{n,\Acal^{\otimes r}}^\ast)$ 
such that for any field extension $k'$, $p$-ring $R'$ with residue field $k'$, and fraction field $K'$,  we have a commutative
diagram:
\begin{equation} \label{diag:gen}
\xymatrix{ 
\SKb_1(A)(k')  \ar[rr]^{\tilde{\rho}_{k'}} & & H^4_{n,A^{\otimes r}}(k') \ar[d] \\
\ar[u]^{\cong} \SKb_1(B_K)(K') \ar[rr]_{\rho_{K'}} 
 & & H^4_{n,B_K^{\otimes r}} (K').
}
\end{equation}
\end{thm}

\begin{rem} \label{rem:multmod}
The invariants obtained by this theorem are the \textit{wild generalisations} of their moderate variants (hence the terminology \textit{moderate versus wild}).  If $\rho$ is a moderate invariant of $\SKb_1$, we denote the wild generalisation
by $\tilde{\rho}$.  If $A$ is a central simple $k$-algebra of $\ind_k(A)\in k^\times$ (with $\car(k)=p>0$), it is in general not clear
whether $\rho_A=\tilde{\rho}_A$.  By the uniqueness of the theorem, to prove such an equality it suffices to verify that $\rho$ satisfies a lifting property as in \eqref{diag:gen}.
\end{rem}

\begin{rem}[\textit{for the reader who takes the effort to look at the original paper.}] \label{rem:l}
In the original statement, the author treats just the case $r=1$.
The proof does not depend on $r$, so it can easily be generalised to any $r$.
If $r=0$, we can also use (ibid., Cor. 4.14) straightaway  to prove the
theorem.  Also an extra field extension $L$ of $k$ is used.  This is to be sure $\Hcal^\ast_{n,L}$ of (ibid., Def. 4.3)
is a cycle module with base $R$.  We do not explicitly need this here.  Even more, the statement over here
is not weaker as by functoriality any invariant  has  images in $\Hcal^\ast_{n,L,\Acal^{\otimes r}}$.
\end{rem}

\begin{rem}
 Note that the theorem actually defines an injective morphism
\begin{equation} 
\Inv^4(\SKb_1(B_K),\Hcal_{n,\Bcal_K^{\otimes r}}^\ast) \to \Inv^4(\SKb_1(A),\Hcal_{n,\Acal^{\otimes r}}^\ast).  
\end{equation}
 As the invariants $\rhokahn$ and $\rhosus$ are non-trivial for biquaternion algebras in characteristics different from 2, this induces their wild generalisations 
to be non-trivial for biquaternion algebras in characteristic 2.  
\end{rem}

\section{Biquaternion algebras} \label{sec:biquat}

In \cite[\S 17]{kmrt}, Knus-Merkurjev-Rost-Tignol construct an invariant of the reduced Whitehead group
of biquaternion algebras in any characteristic.  For sake of brevity we call it \textit{KMRT's invariant}.  
If the characteristic of the base field is not equal to 2, it is known that 
this invariant essentially equals Suslin's invariant.  In this section, we prove in the characteristic 2 case 
it is essentially equal to Suslin's generalised invariant.

\subsection{Definition}

We start by giving the concrete definition of KMRT's invariant.  This needs the notion of involutions
on Azumaya algebras and Witt groups.

\subsubsection{Involutions on Azumaya algebras}

In order to define the invariant, a symplectic involution $\sigma$ on the biquaternion algebra is used.
We  recall the definition of a symplectic involution on an Azumaya algebra (so in particular on a 
central simple algebra).  We treat this in this general setting of Azumaya algebras, because we  need this for our purposes later on.  We refer to \cite[Ch. III, \S 8]{Knus} for more details
on involutions on Azumaya algebras.

\begin{defin}
Let $R$ be a ring and $A$ an Azumaya algebra over $R$ with an $R$-linear involution $\sigma$. 
Suppose
$\alpha:A\otimes_R S \overset{\sim}{\to} M_n(S)$ is a faithfully flat splitting of $A$.  Then $\tilde{\sigma}=\alpha (1\otimes \sigma) \alpha^{-1}$
is an involution on $M_n(S)$.  Since $x \mapsto \tilde{\sigma}(x^t)$ is an automorphism of $M_n(S)$, we can 
choose $u\in \GL_n(S)$ such that $\tilde{\sigma}(x)=ux^tu^{-1}$ for all $x\in M_n(S)$.  Because $\tilde{\sigma}^2=1$, we get
$u^t=\epsilon u$ for some $\epsilon\in \mu_{2}(S)$.  Then $\epsilon$ is called the type of $\sigma$ (it is well defined
and independent of the choice of a faithfully flat splitting \cite[Ch. III, 8.1.1.]{Knus}).  An involution of type 1
is called \textit{orthogonal} and an involution of type -1 is called \textit{symplectic}.
\end{defin}

\begin{rem}
If $R$ is an integral domain, then an involution on an Azumaya algebra can only have type $1$ or $-1$.  When $k$ is a field,
a central simple $k$-algebra of odd degree can only have orthogonal involutions, while a central simple algebra of even degree
can have involutions of both types \cite[Cor. 2.8]{kmrt}.
\end{rem}

If $A$ is a central simple algebra over $k$ of degree $2n$ with a symplectic involution $\sigma$, we can refine the definition of
reduced norm, trace and characteristic polynomial.  Indeed, if $a\in \Symd(A,\sigma)=\{ a + \sigma(a) \, | \, a\in A\}$, the reduced characteristic polynomial $\Prd_{A,a/k}(X)$ is a square \cite[Prop. 2.9]{kmrt}.  Take
$\Prp_{\sigma,a/k}(X)$ the unique monic polynomial such that $\Prd_{A,a/k}(X)=(\Prp_{\sigma,a/k}(X))^2$; this is the \textit{Pfaffian characteristic polynomial}.  The \textit{Pfaffian trace $\Trp_{\sigma/k}(a)$} and 
the \textit{Pfaffian norm $\Nrp_{\sigma/k} (a)$} are
defined as coefficients of $\Prp_{\sigma,a/k}(X)$, compatible with the expression of $\Nrd_{A/k}(a)$ and $\Trd_{A/k}(a)$
as coefficients of $\Prd_{A,a/k}(X)$ (see standard notations in \S \ref{sec:intro}):
\[ \Prp_{\sigma,a/k}(X)=X^n-\Trp_{\sigma/k}(a)X^{n-1}+\ldots+(-1)^n\Nrp_{\sigma/k}(a).\]
So $\Nrd_{A/k}(a)=(\Nrp_{\sigma/k} (a))^2$ and $\Trd_{A/k}(a)=2\Trp_{\sigma/k} (a)$.

\subsubsection{Witt groups}
To explain the value group of KMRT's invariant, we need Witt groups and rings.\footnote{Do not mix up the Witt group and Witt ring with $W_n(k)$ consisting
of the Witt vectors on a field $k$ - see \S \ref{sec:generalising} \ref{sec:logdif}.}
The \textit{Witt group} $W_q(k)$ is the group of Witt-equivalence classes of non-singular quadratic spaces over $k$ with addition defined by
the orthogonal sum $\perp$.
The \textit{Witt ring} $W(k)$ is the ring of Witt-equivalence classes  of non-singular symmetric bilinear spaces with addition given by the orthogonal sum
$\perp$ and multiplication by the tensor product $\otimes$.

\begin{rem} 
If $\car(k)\neq 2$, we know that  as groups (with the orthogonal sum) $W_q(k)$ and $W(k)$ are isomorphic; not as rings, since one can not come 
up with a direct definition of multiplication of quadratic forms.  For our purposes we are however interested in the characteristic
2 case, so we have to  make a clear distinction.   For more information on Witt groups and Witt rings in this general case,
we refer to \cite[Ch. I]{baeza} and \cite[Ch. 1]{kahnquadform} (including the discussion on the characteristic
2 case by Laghribi in \cite[App. E]{kahnquadform}).
\end{rem}

We can equip $W_q(k)$ with a $W(k)$-module structure.
If $(V,B)$ is
a non-singular symmetric bilinear space on $k$ and $(V',q)$ is a non-singular quadratic space on $k$, then $(V\otimes V',B\otimes q)$ is a
quadratic space on $k$ with $B\otimes q$ defined by 
\[ (B\otimes q)(v\otimes v')=B(v,v)q(v'), \quad \text{ for } v\in V,v'\in V'.\] 
 Let $I(k)$ be the \textit{fundamental ideal} of $W(k)$ (generated by the non-singular bilinear spaces of even dimension).  For any
integer $n\geq 0$, we set $I^n(k)=(I(k))^n$  (with $I^0(k)=W(k)$) and $I^nW_q(k)=I^n(k)\otimes W_q(k)$.  This clearly defines
a filtration
\[ W_q(k) = I^0W_q(k) \supset I^1W_q(k) \supset I^2W_q(k) \supset \ldots \] 
We denote the graded quotients by $\overline{I^nW_q(k)}=I^nW_q(k)/I^{n+1}W_q(k)$.  
\begin{rem} \label{rem:isequal}
Set $W'_q(k)$ the subgroup of $W_q(k)$ consisting of equivalence classes of even-dimensional non-singular quadratic spaces over $k$ and
$I^nW'_q(k)=I^n(k)\otimes W'_{q}(k)$.  If $\car(k)\neq 2$, we have $I^nW'_q(k)=I^{n+1}(k)$ by the equivalence of symmetric bilinear and quadratic spaces.  Again, in general we are not able to use this fact.
\end{rem}

\subsubsection{Definition}

Suppose $A$ is a biquaternion algebra over $k$ (see \S \ref{sec:platonov} \ref{sec:cyclicalgs}) and
suppose furthermore that $\sigma$ is a symplectic involution on $A$.  Knus-Merkurjev-Rost-Tignol construct an explicit map
\[ \SL_1(A) \to \overline{I^3W'_q(k)}: a \mapsto 
\begin{cases}
0 & \text{ if } \sigma \text{ is hyperbolic,} \\
\Phi_v+I^4W'_q(k) & \text{ if } \sigma \text{ is not hyperbolic,} \\
\end{cases}
\]
with kernel equal to $[A^\times,A^\times]$.  Recall that an involution is called \textit{hyperbolic} if there exists
an idempotent $e\in A$ such that $\sigma(e)=1-e$.  Furthermore, $\Phi_v$ is the quadratic form
\[ A \to k : x \mapsto \Phi_v(x)=\Trp_\sigma(\sigma(x)vx), \]
where $v\in \Symd(A,\sigma)\cap A^\times$ satisfies $v(\Trp_\sigma(v)-v)^{-1}=-\sigma(a)a$.  There always exists a $v$ satisfying this condition 
\cite[Lem. 17.3]{kmrt}.  This definition is well defined and independent of the choice of $v$ and $\sigma$.  Moreover the construction
is functorial so that we get an invariant $\rhokmrta{A}$ of $\SKb_1(A)$.

\subsection{Comparison, moderate case} \label{sec:compmod2}

In this section, we recall why $\rhokmrta{A}$ and $\rhosusa{A}$ are equal if $A$ is a biquaternion
algebra over $k$ with $\car(k)\neq 2$.  This is because both Suslin and Knus-Merkurjev-Rost-Tignol proved
their invariant of $\SKb_1(A)$ equals $\rhorosta{A}$.  We already recalled
the commutative diagram \eqref{eq:modcompb} giving us the equality of  $\rhosusa{A}$ and $\rhorosta{A}$.

To compare $\rhokmrt$ to $\rhorost$, famous isomorphisms are used, most of
them recently proved.  Indeed, there are isomorphisms $\psi_{F}^1:K^M_4(F)/2  \to \overline{I^4(F)}=I^4(F)/I^5(F)$ for any $F$ of $\car(F)\neq 2$ (Milnor's conjecture for quadratic forms \cite[Q. 4.3]{milnor}, proved by Orlov-Vishik-Voevodsky \cite[Thm 4.1]{orlvisvoe}) and  $\psi^2_F:H^4(F,\mu_2) \to K^M_4(F)/2$
(Milnor's conjecture \cite[\S 6]{milnor} or a special case of the Bloch-Kato conjecture \eqref{eq:galoissymbol}).  

So the obvious way of comparing
$\rhokmrt$ and $\rhorost$ is by the composed isomorphism $\psi_F=\psi^1_F\circ \psi^2_F$.  Indeed, Knus-Merkurjev-Rost-Tignol prove that the following diagram commutes \cite[Notes \S 17]{kmrt}:
\begin{equation} \label{eq:modcompa}
\xymatrix{
0 \ar[r] & \SKb_1(A)(F) \ar[d]_{=}\ar[rr]^{\rhorosta{A,F}} & & H^4_2(F) \ar[d]_\cong^\psi \ar[r] & H^4_2(F(Y)) \ar[d]^\cong \\
0 \ar[r] & \SKb_1(A)(F) \ar[rr]^{\ \ \rhokmrta{A,F}} & & \overline{I^4(F)} \ar[r] & \overline{I^4(F(Y))},
}
\end{equation}
for $F$ any field extension of $k$ and $Y$ the Albert form attached to $A$ from \S \ref{sec:susinvariants} \ref{sec:invbiquat}.

So combining \eqref{eq:modcompb} and \eqref{eq:modcompa}, it follows that $\rhosus$ and $\rhokmrt$ are the same for biquaternion
algebras in characteristic different from 2.

\subsection{Comparison, wild case} \label{sec:compsauv}

We first explain how to lift central simple algebras with a symplectic involution.  
 We  do this for general central simple algebras and later on use the result for biquaternion algebras.

\subsubsection{Lifting algebras with involution}

Let $k$ be a field of $\car(k)=p>0$ and $R$ a $p$-ring with residue field $k$ and fraction field $K$.
Take  an Azumaya algebra  $A$ over $R$ of degree $2n$ with symplectic
involution $\sigma$.  Define the $R$-group scheme $\PGSpB(A,\sigma)=\AutB (A,\sigma)$, defined for any $R$-algebra $S$ by 
\[ \AutB (A,\sigma)(S)= \Aut (A_S,\sigma_S) = \{ \varphi \in \Aut_S (A_S) \, | \, \varphi \circ \sigma_S = \sigma_S \circ \varphi \},\]
with $\sigma_S=\sigma \otimes \id$ the canonical extension of $\sigma$ to $A_S$.
All Azumaya algebras of degree $2n$ with symplectic involutions up to isomorphism are classified by 
$H^1_{\text{\'et}}(R,\PGSpB(A,\sigma))$ \cite[29.22]{kmrt}. Since $\PGSpB(A,\sigma)$ is a smooth group scheme
(proof as in the field case \cite[p. 347]{kmrt}), we can use Hensel's lemma \`a la Grothendieck to get an isomorphism \cite[Exp. XXIV, Prop. 8.1]{sga33}:
\[ H^1_{\text{\'et}}(R,\PGSpB(A,\sigma)) \cong H^1(k,\PGSpB(\Abar,\bar{\sigma})),\]
where $\Abar=A \otimes_R k$ is the reduced central simple $k$-algebra and $\bar{\sigma}=\sigma \otimes \id$ is the reduced
involution on $\Abar$, which is also symplectic.
On the other hand, we have an inclusion
\[ H^1_{\text{\'et}}(R,\PGSpB(A,\sigma)) \hookrightarrow H^1(K,\PGSpB(A_K,\sigma_K)).\]
So in total, we have an inclusion 
\[ H^1(k,\PGSpB(\Abar,\bar{\sigma})) \hookrightarrow H^1(K,\PGSpB(A_K,\sigma_K)). \]

\begin{rem}
Note that this lift coincides with lifting central simple algebras as explained in \S \ref{sec:generalising} \ref{sec:lifts}. 
Over there we actually used the same arguments for the smooth $R$-group scheme $\PGLb_{R,\infty}$ in order to prove \[ \Br(k)=H^1(k,\PGLb_{k,\infty}) \hookrightarrow H^1(K,\PGLb_{K,\infty})=\Br(K) .\]
\end{rem}

So starting with a central simple $k$-algebra $A$ with symplectic involution $\sigma$, we find a lifted
Azumaya algebra $B$ over $R$ with symplectic involution $\tau$ and hence a central simple
$K$-algebra $B_K$ with symplectic involution $\tau_K$.  In particular, $\deg_k(A)=\deg_K(B_K)$ and $\per_k(A)=\per_K(B_K)$.
Since biquaternion algebras are exactly the central simple algebras of degree 4 and period 1 or 2, we see that
a biquaternion algebra over $k$ with symplectic involution lifts to a biquaternion algebra  with symplectic involution over $K$.

\subsubsection{Preparing the ingredients}

We  now continue the work of \S \ref{sec:compmod2} in the wild case.  
Throughout this section, let $k$ be a field of characteristic 2, $R$ a $2$-ring with residue field $k$
and fraction field $K$, and  $A$ a biquaternion algebra over $k$ with lifted Azumaya algebra $B$ over
$R$.
As $\rhosusg$ and $\rhokmrt$ have different value groups, we first give some remarks on how they relate and how
we can use the uniqueness statement of Theorem \ref{thm:lift} to compare the invariants.

By a theorem of Kato, we have an isomorphism 
$\psi_k:H^{4}_2(k)\to \overline{I^3W_q(k)}$ \cite{katoquadform}.
Similar to Suslin's construction \eqref{eq:modcompb}, we can also give a morphism
$H^4_{4,A}(k) \to H^4_2(k)$.
Indeed, the \textit{projection} 
\[ \pi^2_1: W_2(k)\to W_1(k): (a_0,a_1) \to (a_0) \]
gives a morphism $r:H^{4}_4(k)\to H^{4}_2(k)$.  Since $\pi^2_1$ sends elements of order 2 to 0, $r$ does exactly the same.  Hence we get a morphism $r_A:H^{4}_{4,A}(k)\to H^{4}_2(k)$ because any element of
$ K^M_{2}(k) \cdot [A]$ is of order 2.
Now we can compare the different groups with a commutative diagram.

\begin{prop} \label{prop:diagcomp}
Let $k'$ be a field extension of $k$ and $R'$ a $2$-ring (containing $R$) with residue field $k'$ and fraction field $K'$,
then the following diagram commutes: 
\begin{equation} \label{eq:diagcomp}
\xymatrix{
H^4_{4,A}(k') \ar[d]^{i^\ast} \ar[r]^{r_A}  & H^4_2(k') \ar[r]_{\cong\ \ }^{\psi_{k'}\ } \ar[d]^{i^\ast} & \overline{I^3W_q(k')} \ar[d]^{j} \\
H^4_{4,B_K}(K') \ar[r]_{\  r_{B}}  & H^4_2(K') \ar[r]^{\cong \ \ }_{\psi_{K'}\ } & \overline{I^3W_q(K')}. 
}
\end{equation}
\end{prop}

\begin{rem} \label{rem:definj}
The morphisms $r_{B}=r_{B_{K'}}$ and $\psi_{K'}$ are as in \eqref{eq:modcompb} and \eqref{eq:modcompa}, while $r_A=r_{A_{k'}}$ and $\psi_{k'}$ are as above.  
The morphism $j$ on Witt groups is as in \cite[Ch. V, Cor. 1.5]{baeza}; it is the composition of a bijection
of $W_q(R')\cong W_q(k')$ induced by the residual morphism $R'\to k'$ and an injection $W_q(R')\to W_q(K')$.  Here $W_q(R')$ is
the Witt group of quadratic spaces of \textit{constant rank} over $R'$. See \cite[Ch. I and V]{baeza} for more information.

The maps $i^\ast$ are defined by Kato as in \S \ref{sec:generalising} \ref{sec:lifts}.  We recall the exact definition of this morphism which we  need
in the proof: for any integer $n>0$
\[ i^\ast: H^{q+1}_{2^n}(k') \to H^{q+1}_{2^n}(K') \text{ is defined by } w \otimes \bar{b}_1 \otimes \ldots \otimes \bar{b}_q \mapsto i(w) \cup h^q_{2^n,K'}(\{b_1,\ldots ,b_q\}), \]
where $b_1,\ldots,b_q\in R'$, the morphism $h^q_{2^n,K'}$ is the Galois symbol \eqref{eq:galoissymbol} and $i(w)$ is the composition
\[ W_n(k')/\{ w^{(p)}-w | w \in W_n(k')\} \cong H^1(k',\mathbb{Z}/2^n\mathbb{Z}) \overset{\iota}{\hookrightarrow} H^1(K',\mathbb{Z}/2^n\mathbb{Z}),\]
where the isomorphism is induced by the additive form of Hilbert 90 for $W_n(k_s')$ applied to Witt's short exact sequence 
\cite[\S 5]{witt}:
\[ \xymatrix{ 0 \ar[r] & \mathbb{Z}/2^n\mathbb{Z} \ar[r] & W_n(k'_s) \ar[r]^{x^{(2)}-x} & W_n(k'_s) \ar[r] & 0. }\]
The injection $\iota$ is defined in a similar way as one can get an injection from the splitting \eqref{eq:splitting}.
It can be proved that $i^\ast$ behaves well by going to the relative cohomology groups $H^4_{4,A}(k')$ and $H^4_{4,B_K}(K')$ \cite[Prop. 4.10]{wsuspos}.
\end{rem}

\begin{proof}
Let $R'_{\text{nr}}$ be a $2$-ring with residue field $k'_s$ and fraction field $K'_{\text{nr}}$.  So $R'_{\text{nr}}$ is the integral closure of $R'$ in $K'_{\text{nr}}$.

We  first prove $i^\ast \circ r_A=r_B \circ i^\ast$.  This follows merely by the definition of $i^\ast$.
Let $(a_0,a_1)\otimes x_1 \otimes x_2 \otimes x_3 \in H^4_{4,A}(k')$ and take 
$(b_0,b_1)\in W_2(k'_s)$ such that
$(b_0^2,b_1^2)-(b_0,b_1)=(a_0,a_1)$.
Then $(a_0)=(b_0)^2-(b_0)\in W_1(k')$ and
\[ i^\ast \circ r_{A} ((a_0,a_1)\otimes x_1 \otimes x_2 \otimes x_3)=
(\bar{\sigma}(b_0)-b_0)_{\sigma \in \Gamma_{K'}}
\cup h^3_{2}(\{x_1,x_2,x_3\}),\]
where we consider $\bar{\sigma}(b_0)-b_0$ as an element of $\Zb/2\Zb$ for any $\sigma\in \Gamma_{K'}$ (with residue
$\bar{\sigma}\in \Gamma_{k'}$).
On the other hand, 
\begin{eqnarray*} 
r_{B} \circ i^\ast((a_0,a_1)\otimes x_1 \otimes x_2 \otimes x_3)&=&
r_B\left[(\bar{\sigma}(b_0,b_1)-(b_0,b_1) )_{\sigma \in \Gamma_{K'}} \cup h^3_{4}(\{x_1,x_2,x_3\})\right] \notag \\ 
& =& (\bar{\sigma}(b_0)-(b_0) )_{\sigma \in \Gamma_{K'}} \cup h^3_{2}(\{x_1,x_2,x_3\}).
\end{eqnarray*}

The commutativity of the right square is essentially due to Kato \cite[Lem. 11]{katoquadform};  he proves
the existence of a commutative diagram
\[ 
\xymatrix{ 
H^n_2(k')\ar[r]^\cong \ar[d]^\varphi & \overline{I^3W_q(k')} \ar[d]^j \\
K_n^M(K')/2K_n^M(K') \ar[r]_{\qquad \psi^1_{K'}}^{\qquad \cong} & \overline{I^3W_q(K')} 
}
\]
where $\psi^1_{K'}$ is the isomorphism of Milnor's conjecture on quadratic forms (see \S \ref{sec:compmod2}) and $\varphi$ is defined by
\[ \bar{b} \frac{d\bar{a}_1}{\bar{a}_1}\wedge \frac{d\bar{a}_2}{\bar{a}_2}\wedge  \frac{d\bar{a}_3}{\bar{a}_3} \mod I \mapsto
 \{ 1+4b,a_1,a_2,a_3\} \mod 2K_n^M(K'),
\]
for $a_1,a_2,a_3,b\in R'$.
Since the isomorphism $\psi_{K'}:H^4_2(K')\to\overline{I^3W_q(K')}$ is defined as composition of $\psi^1_{K'}$ with the Galois symbol
$h^4_{2,K'}$, it suffices to check $i(\bar{b})=h^1_{2,k'}(4b+1)$ for any $b\in R'$.
So take $c\in k_s'$ such that $c^2-c=\bar{b}$.  Then 
\[i(\bar{b})=(\bar{\sigma}(c)-c)_{\sigma \in \Gamma_{K'}}\in H^1(K',\Zb/2)\] under the standard identification
of $\Zb/2$ and $\mu_{2}(K')$.  
Take $\tilde{c}$ to be a lift of $c$ in $R_{\text{nr}}$.  By eventually changing  the representant of $\bar{b}$ in $R'$, we can assume
$\tilde{c}^2-\tilde{c}=b$.  Then $4b+1=(2\tilde{c}+1)^2$ and 
\[ h^1_{2,K'}(4b+1)=(\sigma(2\tilde{c}+1)/(2\tilde{c}+1))_{\sigma \in \Gamma_{K'}} \in H^1_2(K').\]
So if $\sigma(2\tilde{c}+1)/(2\tilde{c}+1)=1$, we have $\sigma(\tilde{c})=\tilde{c}$.  On the other hand, if
$\sigma(2\tilde{c}+1)/(2\tilde{c}+1)=-1$, we get $\sigma(\tilde{c})=-\tilde{c}-1$.  This gives indeed the desired
equality.
\end{proof}

\subsubsection{Cooking up the result}

Using Theorem \ref{thm:lift} and Proposition \ref{prop:diagcomp}, we can prove the main theorem.

\begin{thm} \label{thm:kmrtlift}
Let $k$ be a field of characteristic 2 and $A$ a biquaternion algebra over $k$, then for any
field extension $k'$ of $k$
\[ \rhokmrta{A,k'} = \psi_{k'} \circ r_A \circ \rhosusga{A,k'} \]
with $\psi_{k'}$ and $r_A$ as in \eqref{eq:diagcomp}.
\end{thm}

\begin{proof}
Let $k'$ be a field extension of $k$  and $R$ (resp. $R'$) a $2$-ring with residue field $k$ (resp. $k'$)
and fraction field $K$ (resp. $K'$).
Suppose $\sigma$ is a symplectic involution on $A$
and take $B$ a lifted Azumaya $R$-algebra with lifted symplectic involution $\tau$.
Use the same notations as in \eqref{eq:diagcomp}.
We know $j$ is injective,  
$i^\ast \circ \rhosusga{A}=\rhosusa{B_K}$ (by definition of $\rhosusga{A}$) and $\rhokmrta{B_K}=\psi_{K'} \circ r_B \circ \rhosusa{B_K}$.
So it suffices to prove that $\rhokmrta{B_K}=j\circ \rhokmrta{A}$, which  merely follows from the definition.

Let us first explain the isomorphism $\SKb_1(B_K)(K')\cong \SKb_1(A)(k')$.
We can suppose that $\SKb_1(A)(k')\neq 0$ so that $A_{k'}$ and $B_{K'}$ are division 
algebras by Wang's theorem \cite{wang}.  Then $B_{K'}$ is equipped with a valuation $w$ that extends the valuation $v'$ of $K'$, namely $\frac{1}{4}v'\circ \Nrd_{B_{K'}/K'}$.  The associated valuation ring is $B_{R'}$ and the reduced $k$-algebra is $A_{k'}$.  Even more, $\SL_1(B_{K'})$ is part of $B_{R'}$ and
the isomorphism $\SKb_1(B_K)(K')\cong \SKb_1(A)(k')$ is induced by the residue map on $\SLb_1(B_{K'})$ \cite[Cor. 3.13]{sk1niettriveng} -- see also \cite[Cor. 3.3]{wsuspos}.  

The involutions $\sigma$ and $\tau$ can not be hyperbolic due to \cite[Prop. 6.7 (3)]{kmrt}.  Take
$a\in \SKb_1(A)(k')$ with lift $b \in \SKb_1(B_K)(K')$.  Then by definition it follows that $\Prd_{A_{k'},a/k'}(X)=\overline{\Prd_{B_{K'},b/K'}(X)}$, where the residue is the canonical residue on $R'[X]$.  So we also get
$\Prp_{\sigma_{k'},a/k'}(X)=\overline{\Prp_{\tau_{K'},b/K'}(X)}$ and $\Trp_{\sigma_{k'}/k'}(a)=\overline{\Trp_{\tau_{K'}/K'}(b)}$.
Then take $y\in \Symd(B_{K'},\tau_{K'})\cap B_{K'}^\times$ satisfying $y(\Trp_{\tau_{K'}/K'}(y)-y)^{-1}=-\tau(b)b$.
We can assume $w(y)\geq 0$, since if $w(y)<0$, i.e. $\Nrd_{B_{K'}/K'}(y)=\lambda/\mu \in K'$ with $\lambda,\mu\in R'$, then 
$w(\mu y)=v(\lambda)\geq 0$ and 
\[\mu y\left(\Trp_{\tau_{K'}/K'}(\mu y)-\mu y\right)^{-1} = y(\Trp_{\tau_{K'}/K'}(y)-y)^{-1}. \]
Then we get $\bar{y}(\Trp_{\sigma_K'/K'}(\bar{y})-\bar{y})^{-1}=-\sigma(a)a$
as $b$ is a lift of $a$ and moreover $\bar{y}\in \Symd(A,\sigma)$.
Hence 
\begin{eqnarray*}
\rhokmrta{A,k'}(a)&=&\Phi_{\bar{y}}:A_{k'}\to k': x \mapsto \Trp_{\sigma_{k'}/k'}(\sigma_{k'}(x)\bar{y}x)\quad \text{ and}\\
\rhokmrta{B_{K'},K'}(b)&=&\Phi_{y}:B_{K'}\to K': x \mapsto \Trp_{\tau_{K'}/K'}(\tau_{K'}(x)yx).
\end{eqnarray*}
As $\overline{\Trp_{\tau_{K'}/K'}(\tau_{K'}(x)yx)}=\Trp_{\sigma_{k'}/k'}(\sigma_{k'}(\bar{x})\bar{y}\bar{x})$ for $x\in B_{R'}$,
the required compatibility holds.
\end{proof}

\subsection{Non-triviality of the invariant} \label{sec:nontrivbiquat}

Because the invariants for biquaternions in odd or zero characteristic are injective, they are also injective in
characteristic 2 due to the lifting property (Theorem \ref{thm:lift}).  As $\SKb_1$ is not trivial for Platonov's examples (\S \ref{sec:platonov} \ref{sec:nontrivsk1}) and in general for biquaternion algebras of index 4 \cite{mersuslinbiquat}, we find non-trivial invariants
in characteristic 2.

Another argument for non-triviality of $\rhokmrt$ in characteristic different from 2 is given by a formula of Merkurjev 
for the value on the centre of the biquaternion algebra \cite[Ex. p. 70]{merkrostinv} -- see also \cite[Ex. 17.23]{kmrt}.
Using this formula and the lift from characteristic 2 to characteristic 0, one
could hope to prove the non-triviality of $\rhokmrt$ (and hence of $\rhosus$) in the case when $\car(k)=2$, but this fails.  
Let us comment on this fact.

Say $k$ is a field of characteristic 2, $R$ a $p$-ring with residue field $k$ and fraction field $K$, and let $A=[\bar{a},\bar{b})\otimes_k [\bar{c},\bar{d})$ be a biquaternion $k$-algebra for $a,c\in R$
and $b,d\in R^\times$.  Then the lifted Azumaya $R$-algebra is $B=[a,b)\otimes_R[c,d)$ where e.g. $[a,b)$ is the $R$-algebra
generated by $u,v$ satisfying slightly different relations than usual: $u^2+u=a$, $v^2=b$, and $uv=-v(u+1)$.  We can rewrite it as $B=(4a+1,b)_R\otimes_R (4c+1,d)_R$, where
$(4a+1,b)_R$ is the $R$-algebra generated by $i,j$ with $i^2=4a+1$, $j^2=b$, and $ij=-ji$. Indeed, an isomorphism is given by 
$i=2u+1$ and $j=v$.
Suppose $K$ contains a primitive fourth root of unity $\zeta$, then by (loc. cit.) we have
\[ \rhokmrta{B_K,K}([\zeta])=\pform{ 4a+1,b,4c+1,d} + I^4W'_q(K),\]
where  $[\zeta]$ is the class of $\zeta$ in $\SKb_1(B_K)(K)$ and where $\pform{4a+1,b,4c+1,d}$ is an \textit{$n$-fold Pfister quadratic $K$-form} \cite[Lem. 2.1.1]{kahnquadform}. 

Let $\pi$ be the isomorphism $\SKb_1(B_K)(K)\cong \SKb_1(A)(k)$, then  $\pi([\zeta])=[1]$ because $k$ contains no non-trivial
fourth roots of unity.  By the proof of Theorem \ref{thm:kmrtlift}, we have $j\circ \rhokmrta{B_K,K}([\zeta]) = \rhokmrta{A,k}\circ \pi ([\zeta]) = 0 \in \overline{I^3W'_q(k)}$.  Because the map
$j$ from Proposition \ref{prop:diagcomp} is injective, we get that 
$\pform{4a+1,b,4c+1,d}=0\in \overline{I^3W'_q(K)}$.  We can also verify this by calculating with Pfister forms.
Define $\Qcal=(4a+1,b)_R$ and let $\Xcal$ be the natural affine $R$-scheme with 
\[ \Xcal(R)=\{ x\in \Qcal \, | \, \Nrd_{\Qcal_K/K}(x)=4c+1 \},
\]
where $\Qcal_K=\Qcal \otimes_R K$.  Then $\Xcal$ is an $R$-torsor under $\SLb_1(\Qcal)$, where $\SLb_1(\Qcal)$ is the natural affine $R$-scheme so that $\SLb_1(\Qcal)(R)=\SLb_1(\Qcal_K)(K)\cap \Qcal$.
The special fibre $\Xcal_k=\Xcal \times_R k$ clearly has a rational point, so its class $[\Xcal_k]\in H^1(k,\SLb_1(\Qcal_k))$ is
trivial.  By Hensel's lemma \cite[Exp. XXIV, Prop. 8.1]{sga33}, we get $[\Xcal]=0\in H^1_{\text{\'et}}(R,\SLb_1(\Qcal))$.  Hence $\Xcal$ (as well as the generic fibre $\Xcal_K$) has a rational point, but then by theory of Pfister forms we get $\pform{4a+1,b,4c+1}=0\in W_q'(K)$ \cite[Cor. 2.1.10]{kahnquadform}.  Indeed, $\Nrd_{\Qcal_K/K}(x)$ corresponds with a value of $\pform{4a+1,b}$.  So a fortiori
$\pform{4a+1,b,4c+1,d}=0 \in \overline{I^3W'_q(k)}$.

\section{Comparing to Kahn's invariant}  \label{sec:compkahn}

We compare now all defined invariants of $\SKb_1(A)$ to $\rhokahna{A}$ in the moderate case,
i.e. as they are originally defined.  The results can be generalised to the wild invariants, but with some loss of information.  
We also generalise the formula of Merkurjev for the value on the  centre of $\SKb_1(A)$ (\S \ref{sec:nontrivbiquat}).

\subsection{Moderate case} \label{sec:compkahnmod}

We explain two natural ways of comparing $\Inv^4(\SKb_1(A),\Hcal^{\ast}_{n})$ and $\Inv^4(\SKb_1(A),\Hcal^{\ast}_{n,\Acal^{\otimes r}})$. Let $A$ be a central simple $k$-algebra
with $\ind_k(A)=n\in k^\times$ and $m=\per_k(A)$.

\subsubsection{Ways of looking} \label{sec:ways}
For any field extension $F$ of $k$ and any integer $r$,
we can look at the composition 
\[ m_r: H^4_{n,A^{\otimes r}}(F) \overset{\cdot m} \to H^4_{n/m}(F) \hookrightarrow H^4_{n}(F) \]
and at the projection 
\[ \pi_r: H^4_n(F) \to H^4_{n,A^{\otimes r}}(F).\]
These induce respectively maps
\begin{eqnarray*} \tilde{m}_r:\Inv^4(\SKb_1(A),\Hcal^\ast_{n,\Acal^{\otimes r}}) &\to &\Inv^4(\SKb_1(A),\Hcal^\ast_{n}) \quad \text{ and } \\
 \tilde{\pi}_r:\Inv^4(\SKb_1(A),\Hcal^\ast_{n})& \to &\Inv^4(\SKb_1(A),\Hcal^\ast_{n,\Acal^{\otimes r}}).
\end{eqnarray*}
The maps $\tilde{\pi}_r$ where introduced by Kahn \cite[Rem. 11.6]{kahnsk12}, but we rather consider the maps $\tilde{m}_r$ to compare because
of the special definition of Kahn's invariant as generator of the target group.  We could also refine $\tilde{m}_r$ if $H^2(k,\mu_n^{\otimes 2}) \cup r[A]$ has $m'$-torsion for an integer $0\leq m'<m$.  A good comprehension of both maps actually relies, as Kahn mentions, on a good
comprehension of the cup product with the class of $A$ (loc. cit.).

By the cyclicity of $\Inv^4(\SKb_1(A),\Hcal^\ast_{n})$ (\S \ref{sec:defkahn} \ref{sec:cyclic}), we certainly find 
the following relations.

\begin{prop} \label{prop:mult}
Let $A$ be a central simple $k$-algebra of $\ind_k(A)\in k^\times$.  Then for any integer $r$
and any $\rho \in \Inv^4(\SKb_1(A),\Hcal^\ast_{n,\Acal^{\otimes r}})$ there exists an integer 
$d_A\in \Zb/\overline{n}$ such that 
\[ \tilde{m}_r(\rho)=d_A\, \rhokahna{A} \in \Inv^4(\SKb_1(A),\Hcal_n^\ast)\subset \Zb/\overline{n}.\]
\end{prop}

\begin{proof}
Use the definition of $\rhokahn$ and the bounds on $\Inv^4(\SKb_1(A),\Hcal^\ast_n)$ (see \S \ref{sec:defkahn} \ref{sec:bound}).
\end{proof}

Kahn also raises the issue whether $\tilde{\pi}_r$ is surjective or not (loc. cit.).  We can prove it to be non-surjective
for  biquaternion division algebras \`a la Platonov. 
\begin{prop}  \label{thm:notsurj}
Let $k=\mathbb{Q}_p((t_1))((t_2))$ for a prime $p$.  Suppose $A=(a,t_1)\otimes(b,t_2)$ is a biquaternion division $k$-algebra  
for $a,b\in \Qb_p^\times$. 
Then $\tilde{\pi}_1$ is not surjective.
\end{prop}

\begin{proof}
In \S\S \ref{sec:platonov} \ref{sec:nontrivsk1} and \ref{sec:galcoh} we saw that  $\SK_1(A)\cong \Zb/2$ and $H^4_4(k)\cong \Zb/4$.  
We can also add a fourth primitive root of unity to $k$ as this does not change the Brauer group.  In this
case we have the Bloch-Kato isomorphism $H^4_{4}(k)\cong K^M_4(k)/4$. 

We now prove $H^4_{4,A}(k) \cong \Zb/2$.   
Under the Bloch-Kato-isomorphism $K^M_2(k)/2\cong \phantom{ }_2 \Br(k)$,
the class of $A$ corresponds to $\{a,t_1\}+\{b,t_2\}\in K^M_2(k)/2$ \cite[Prop. 4.7.1]{gilleszam} so that
$H^2(k,\mu_4^{\otimes 2}) \cup [A]$ is isomorphic to $(K^M_2(k)/4) \cdot (2\{a,t_1\}+2\{b,t_2\})$.  As the
isomorphism $H^4_4(k)\cong \Zb/4$ is retrieved by taking two residues $\partial^3_{t_2} $ and $ \partial^4_{t_1}$, it suffices to determine the group 
\[\partial^3_{t_2} \circ \partial^4_{t_1} \bigl( (K^M_2(k)/4) \cdot (2\{a,t_1\}+2\{b,t_2\})\bigr).\]  By the definition of residues
on Milnor $K$-groups \cite[\S 2]{milnor}, it is clear that this equals $(K^M_1(k)/4) \cdot 2\{a\} +  (K^M_1(k)/4) \cdot  2\{ b\}$.
As we assumed that $\SKb_1(A)$ is not trivial, $a$ can not be a square otherwise $A$ would have been Brauer-trivial.
This means that $(K^M_1(k)/4) \cdot 2\{a\} +  (K^M_1(k)/4) \cdot  2\{ b\} $ is not trivial.  On the other hand it has 2-torsion 
inside $K^M_2(k)/4\cong \Zb/4$ so that indeed $H^4_{4,A}(k) \cong \Zb/2$.

Then $\pi_1:\Zb/4\to \Zb/2$ is the ``modulo 2'' map and $m_1:\Zb/2\to \Zb/4$ is canonical injection.
Suslin proves $\rhosusa{A,k}:\SKb_1(A)(k) \to H^4_{4,A}(k)$ is not trivial \eqref{eq:modcompb},
so it is the identity map on $\Zb/2$.  It is then clear that this can never factor through $H^4_4(k)$ 
so that $\tilde{\pi}_1$ is clearly not surjective.
\end{proof}

\subsubsection{Determining factors} \label{sec:detfactor}

We   prove that for the product of two symbol algebras of degree $n$ the factor $d_A$ appearing in Proposition \ref{prop:mult} only depends on the invariant $\rho$ and the characteristic of $k$.

\begin{prop}
\label{thm:mult}
Let $\rho$ be a moderate invariant of $\SKb_1$ with values in  $\Hcal^4_{\otimes r}$.
Let furthermore $p$ be equal to zero or to any prime and let $m$ be an integer not divisible by $p$.  
Then there exist an integer $i(p,m)\in \Zb/\overline{m^2}$ such 
that for any field $k$ of $\car(k)=p$ containing a primitive $m$-th root of unity $\xi_m$ and for any product $A=(a,b)_m \otimes (c,d)_m$ of two symbol $k$-algebras 
\[ \tilde{m}_r(\rho_{A})=i(p,m)\, \rhokahna{A} \in \Inv^4(\SKb_1(A),\Hcal_{m^2}^\ast)\subset \Zb/\overline{m^2}.\]
\end{prop}

\begin{rem}
 Although $i(p,m)$ is in general not uniquely determined, we can take a canonical representant as
we know $\Inv^4(\SKb_1(A),\Hcal_{m^2}^\ast)$ is cyclic.  This comes down to taking the class in $\Zb/\overline{m^2}$
satisfying the required relation and such that the representant in $\{0,\ldots,m^2-1\}$ is as low as possible.
It also of course depends on the invariant.  We add an index if necessary to stress which invariant is
compared to Kahn's invariant.  Moreover, it also depends on the exact definition of the injection $\Inv^4(\SKb_1(A),\Hcal_{m^2}^\ast)\subset \Zb/\overline{m^2}$.  For the remainder of the paper, we fix this injection.
\end{rem}

\begin{proof}
Take $k$ the prime field of characteristic $p$ and set $k'=k(\xi_m)$ for an $m$-primitive root of unity  $\xi_m\in k_s$.
Denote  by $\Tcal=(t_1,t_2)_m \otimes (t_3,t_4)_m$ the product of two Azumaya symbol algebra over 
$R=k'[t_1^{\pm 1},t_2^{\pm 1},t_3^{\pm 1},t_4^{\pm 1}]$
where $t_1,t_2,t_3,t_4$ are variables and where Azumaya symbol algebras are defined using the same relations
as used for symbol algebras over a field.  Take $K=k'(t_1,t_2,t_3,t_4)$ and $T=\Tcal_K=(t_1,t_2)_m \otimes (t_3,t_4)_m$,  the product of the respective symbol algebras over $K$. 
By Proposition \ref{prop:mult}, we find a unique $d_{T}\in \Zb/\overline{m^2}$
such that
\begin{equation} \label{eq:genpoint}
\tilde{m}_r(\rho_{T})=d_{T}\, \rhokahna{T}.  
\end{equation} 
We  prove $d_T$ only depends on $m$ and $p$.  

So suppose $F$ is a field of characteristic $p$ containing an  $m$-th primitive root of unity so that $k'\subset F$.
Take any product $A=(a,b)_m\otimes (c,d)_m$ of two symbol algebras of degree $m$ over $F$.  Now $A$ can be obtained
from $\Tcal_F=\Tcal \otimes_R F$ by specialising $t_1,t_2,t_3,t_4$ to $a,b,c,d$ respectively.  

Furthermore, $(a,b,c,d)$ defines a $k$-rational point $x$ of $\spec(F[t_1^{\pm 1},t_2^{\pm 1},t_3^{\pm 1},t_4^{\pm 1}])$.
Take $\Ocal_x$ to be the local ring of $\spec(F[t_1^{\pm 1},t_2^{\pm 1},t_3^{\pm 1},t_4^{\pm 1}])$ in $x$ with maximal ideal $M$.  It is clear that
the completion $\hat{\Ocal}_x$ of $\Ocal_x$ with respect to the $M$-adic topology is $F$-isomorphic to 
$R'=F[[u_1,u_2,u_3,u_4]]$ where $u_1=t_1-a,u_2=t_2-b,u_3=t_3-c,$ 
and $u_4=t_4-d$ (see also \cite[Thm. 19.6.4]{ega4}).
Under the isomorphism $\Br(R')\cong \Br(F)$ from \S \ref{sec:generalising} \ref{sec:lifts}, it is clear that $A_{R'}=A\otimes R'$ is an Azumaya $R'$-algebra
mapping  to $A$.  Furthermore, the $F$-isomorphism of $\hat{\Ocal}_x$ with $R'$ gives
an isomorphism $\Br(\hat{\Ocal}_x)\cong \Br(R')$.  In its turn,
this gives an isomorphism $\Br(\hat{\Ocal}_x)\to \Br(F)$ with inverse  given by taking the tensor product over $F$ with $\hat{\Ocal}_x$.  By construction it sends the class of $\Tcal_{\hat{\Ocal}_x}$ to
the class of $A$.

Let $K'=F((u_1))((u_2))((u_3))((u_4))$, then
$A \otimes_F K'$ is Brauer-equivalent to $\Tcal_{\hat{\Ocal}_x} \otimes_{\hat{\Ocal}_x} K' \cong T_{K'}$.  
We find $\SK_1(A_{K'})\cong \SK_1(T_{K'})$ (as in \S \ref{sec:generalising} \ref{sec:lifts}).
Furthermore, 
\eqref{eq:splitting} gives an injection $H^4_{m^2}(F)\to H^4_{m^2}(K')$.
By functoriality of invariants, the diagram
\[
\xymatrix{ 
\SK_1(A) \ar[d]_\cong  \ar[r]^\rho & H^4_{m^2}(F) \ar[d] \\
\SK_1(T_{K'})  \ar[r]_{ \rho} & H^4_{m^2}(K' ) \\
}
\]
commutes both for $\tilde{m}_r(\rho)$ and $\rhokahn$. 
Then by \eqref{eq:genpoint}, we get $\tilde{m}_r(\rho_{A})=d_{T}\rhokahna{A}$. 
\end{proof}

In particular, we find such relations for $\rho=\rhosusoud,\rhosus$, and the $\rho_r$'s.  

\subsubsection{Non-triviality of Kahn's invariants} \label{sec:nontrivkahn}

As mentioned
in Remark \ref{rem:kahnbiquatnontriv}, $\rhokahn$ is not-trivial 
 for biquaternion algebras (of index 4).  We generalise this to the product of two cyclic algebras 
\`a la Platonov (\S \ref{sec:platonov}).  Therefore, we compare $\rhokahn$ to $\rhosusoud$ as 
this invariant is non-trivial for Platonov's examples (\S \ref{sec:susinvariants} \ref{sec:suslin91}).  This means we have to work with $\Hcal^\ast_{n,\Acal^{\otimes 2}}$ for
suitable $n$ and $A$.  (In the same way as in Proposition \ref{thm:notsurj}, these give also examples of non-trivial $\tilde{\pi}_2$.)

\begin{thm} \label{thm:nottriv}
Let $k$ be $p$-adic field containing a $n^3$-th primitive root unity.  Suppose
$A=(a,t_1)_n\otimes (c,t_2)_n$ is a division $k((t_1))((t_2))$-algebra, then $\rhokahna{A}$ is not trivial. 
If $n=q_1\cdot \ldots \cdot q_r$ for  different primes $q_i$, then
\[ \Inv^4(\SKb_1(A),\Hcal^\ast_{n^2}) \cong \Zb/n. \]
Moreover if $n$ is odd, the integer $i_{\text{S91}}(0,n)\in \Zb/\overline{n^2}$ defined in Proposition \ref{thm:mult}
for $\rhosusoud$ is not trivial.  
\end{thm}

\begin{proof}
We know $\SK_1(A)\cong \Zb/n$ by \S \ref{sec:platonov}.  Furthermore $H^4_{n^2}(k)=\Zb/n^2$
 as the results in \S \ref{sec:platonov} \ref{sec:galcoh} hold also when one replaces $\Qb_p$ by 
a finite extension of it.  

To calculate $H^4_{n^2,A^{\otimes 2}}(k)$, we use a analogous argument as in the proof of Proposition \ref{thm:notsurj}.  If $n$ is odd, we also find $H^4_{n^2,A^{\otimes 2}}(k)=\Zb/n$ as in this case $\per_k(A^{\otimes 2})=\per_k(A)$.  If $n$ is even, $\per_k(A^{\otimes 2})=n/2$ so that $H^4_{n^2,A^{\otimes 2}}(k)=\Zb/(2n)$.  
In either case, $m_2:H^4_{n^2,A^{\otimes 2}}(k)\to H^4_{n^2}(k)$ is the canonical injection ($m_2$ is the multiplication by $m$ for $m=n$
if $n$ odd and $m=n/2$ if $n$ even).

Suslin proves $\rhosusouda{A}$ is not trivial (on the field $k$) \cite[Thm. 4.8]{sk1niettriveng}.
If $n$ is odd,  $\rhokahna{A}$ is not trivial (on $k$) by  Proposition \ref{prop:mult} and
hence by definition $i_{\text{S91}}(0,n^2)\neq 0 \in \Zb/\overline{n^2}$.  If $n$ is even, a similar argument as
in the proof of  Proposition \ref{prop:mult} gives the non-triviality of $\rhokahna{A}$
(mutatis mutandis $m$ by $n/2$).

By the bound on the invariant group  (\S \ref{sec:defkahn} \ref{sec:bound}) and a Brauer decomposition of $A$ with
a related decomposition of invariants in primary parts, the isomorphism statement follows.
\end{proof}

\subsection{Wild case}

We  continue the comparison in the wild case.  Using a lift, we can generalise the statement to any central simple
algebra with some loss of information.  This does let us prove a relation between the several $i(p,n)$'s.

Let $A$ be a central simple $k$-algebra of $\ind_k(A)=n$ and $\per_k(A)=m$. We again have morphisms for any integer $r$
\[  \tilde{m}_r:\Inv^4(\SKb_1(A),\Hcal^\ast_{n,\Acal^{\otimes r}}) \to \Inv^4(\SKb_1(A),\Hcal^\ast_{n}) \]
induced by the multiplication for any field extension $F$ of $k$:
\[ m_r: H^4_{n,A^{\otimes r}}(F) \overset{\cdot m} \to H^4_{n/m}(F) \hookrightarrow H^4_{n}(F). \]
Note that we can also define maps $\tilde{\pi}_r$ as in \S \ref{sec:compkahnmod} \ref{sec:ways}.

\begin{prop} \label{thm:compwild}
Let $\rho$ be a moderate invariant of $\SKb_1$ with values in $\Hcal^4_{\otimes r}$.
Suppose $k$ is a field of $\car(k)=p>0$ and let $A=[a,b)_p\otimes [c,d)_p$ be the product of two $p$-algebras over $k$,
 then
\[ \tilde{m}_r(\tilde{\rho}_A) = i(0,p) \rhokahnga{A}. \]
\end{prop}

\begin{proof}
Take a $p$-ring $R$ with residue field $k$ and fraction field $K$.
Remark first that the lifted Azumaya $R$-algebra $B$ of $A$ is (after base extension to $K$) 
a product of two symbol algebras of degree $p$.  This follows from \cite[Prop. 4.7.1 \& Prop. 9.2.5]{gilleszam} 
and the injection $H^2_{p^2}(k)\to H^2_{p^2}(K)$ \cite[Proof of Thm. 1 \& 3]{katogalcoh}.

The result follows  immediately from the injections
\begin{eqnarray*} 
\Inv^4(\SKb_1(B_K),\Hcal_{p^2}^\ast) & \to & \Inv^4(\SKb_1(A),\Hcal_{p^2}^\ast) \quad \text{ and } \\
\Inv^4(\SKb_1(B_K),\Hcal_{p^2,\Bcal_K^{\otimes r}}^\ast)& \to & \Inv^4(\SKb_1(A),\Hcal_{p^2,\Acal^{\otimes r}}^\ast)
\end{eqnarray*}
defined by lifting invariants and the relations for $\rho_{B_K}$ and $\rhokahna{B_K}$ (Proposition \ref{thm:mult}).
\end{proof}

\begin{rem}
In the view of Remark \ref{rem:multmod}, we could even refine the statement in the moderate case.  
If $k$ is a field of $\car(k)=p>0$ and if $A=(a,b)_n\otimes (c,d)_n$ is the product of
two symbol algebras for $n\in k^\times$, then a similar statement holds as $A$ lifts to 
the product of two symbol algebras of degree $n$ in characteristic 0. If $\tilde{\rho}_A=\rho_A$, then $i(p,n)$ is a multiple of $i(0,n)$ in $\Zb/\overline{n}$. Indeed, $\rhokahna{A}$ is a generator of $\Inv^4(\SKb_1(A),\Hcal^\ast_n)\subset \Zb/\overline{n}$ and
for some integer $d$
\[ i(p,n) \rhokahna{A} = \tilde{m}_r(\rho_A) = i(0,n)\, \rhokahnga{A} = i(0,n)\, d\, \rhokahna{A}. \]
\end{rem}

\subsection{Formula on the centre}

We can now generalise the formula of Merkurjev on the centre of a biquaternion algebra (\cite[Ex. p.70]{merkrostinv} -- see also \cite[Ex. 17.23]{kmrt} and \S \ref{sec:nontrivbiquat}) to the tensor product of two symbol algebras.
We first prove a general formula and
later we prove a finer result using Theorem \ref{thm:nottriv}.

\subsubsection{General result} We  again use cohomological invariants, however not invariants of algebraic
groups as in \S \ref{sec:merkurjev}, but rather invariants as introduced in \cite[Ch. I]{cohinv}.  
These are also natural transformations of functors, but rather a transformation of a functor $B:\kfields \to \Sets$
into a functor $H:\kfields \to \Groups$.  

\begin{prop} \label{prop:centre}
Let $p$ be equal to 0 or to any prime and let $n>0$ be an integer not divisible by $p$.  Let 
 $\varphi$ be the canonical map $H^4_{m}(k)\to H^4_{n^2}(k)$ (for $m=\overline{n^2}$).
There exists an integer $j(p,n)$ such that the following formula holds for any field $k$ of $\car(k)=p$ containing a primitive $n^2$-th root of unity $\zeta$ 
and for any product $A=(a,b)_n\otimes (c,d)_n$ of two symbol $k$-algebras:
 \[ \rhokahna{A,k}([\zeta]) = \varphi \left[ j(p,n)\, h^4_{m}(\{ a,b,c,d\})\right] \in H^4_{n^2}(k) \]
\end{prop}

\begin{rem}
Remark that as 
$k$ contains an $n^2$-th primitive root of unity, $\mu_{n^2}^{\otimes i}\cong \Zb/n^2$ for any $i> 0$.
Note also that  $\varphi \bigl[ h^4_{m}(\{ a,b,c,d\})\bigr] = (n/m)\, h^4_{n^2}(\{ a,b,c,d\})$.

This expression is compatible to the
biquaternion case keeping in mind diagrams \eqref{eq:modcompb} and \eqref{eq:modcompa}.
Also, the integer
$j(p,n)$ in the theorem is not uniquely determined, but can be picked canonically by taking the smallest positive integer
satisfying the relation.  Moreover, $j(p,n)$ depends on the $n$-th primitive root of unity used in the definition of the
symbol algebra and of the choice of $n^2$-th primitive root of unity $\zeta$.  We are interested in the invertibility 
of $j(p,n)$ modulo $m$ and therefore the exact choices do not matter, so we do not incorporate these in the notation.
\end{rem}

\begin{proof}
As $\rhokahn$ has $m$-torsion (Lemma \ref{lem:boundgen}), we can assume $\rhokahna{A,k}([\zeta])$ to have values in
$H^4_{m}(k)$.

Let  $k$ be the prime field of characteristic $p$ and set $k'=k(\zeta)$ for $\zeta\in \bar{k}$ a primitive $n^2$-th root of unity.
Take  $T=(t_1,t_2)_n\otimes (t_3,t_4)_n$ over $F=k'(t_1,t_2,t_3,t_4)$.
We  prove the formula for $T$.  The proof ends by specialising to $A$ as in the proof of Proposition
\ref{thm:mult}.

Let $B:\kfields \to \Sets$ be the functor attaching to a field extension $F$ of $k$
the Galois cohomology group $H^1(F,\mu_m)^4$ and $H$ associating 
$H^4(F,\mu_m^{\otimes 4})$ with $F$.  
Now $\rhokahn$ induces a cohomological invariant of $B$ into $H$.  Indeed,
using the isomorphism $H^1(F,\mu_n)\cong F^\times/(F^\times)^n$, we associate with any four representants
$a,b,c,d\in F^\times$ of classes in $H^1(F,\mu_m)$ the value
$\rhokahna{A,F}([\zeta]) \in H^4_{m}(F)\cong H^4(F,\mu_{m}^{\otimes 4})\cong K^M_4(F)/m$ (for $A=(a,b)_n\otimes (c,d)_n$).  

Using a full description of all possibles invariants of $B$ into $H$ of \cite[Prop. 2.1 \& \S 3.1]{garibaldi} and
\cite[Ex. 16.5]{cohinv}, we find that 
$\rhokahna{T,F}([\zeta])$ can be written in $K_4(F)/m$ as sum of pure symbols of the form $\lambda \{z_1,z_2,z_3,z_4\}$ where $\lambda$ is an integer
and each $z_i$ is either a $t_j$ either an element of $k$.  We  prove that only $\{t_1,t_2,t_3,t_4\}$
occurs.  By specializing $t_1$ to $1$, we obtain $T_1=(1,t_2)_n\otimes (t_3,t_4)_n$
from $T$.  But then $\SKb_1(T_1)=0$ by Wang's theorem so that $\rhokahna{T_1,F}([\zeta])=0$.  This induces that 
for all (non-trivial) pure symbols $\{z_1,z_2,z_3,z_4\}$ appearing in $\rhokahna{T,F}([\zeta])$ one of the 
$z_i$ has to equal $t_1$ (as the other ones are zero by the specialisation above). Three other specialisations
give the result.
\end{proof}

\begin{rem}
In wild characteristics (i.e. when $p\,|\, n$), a formula as above does not make sense as there
are no non-trivial $p^2$-th roots of unity.  So similar as in \S \ref{sec:nontrivbiquat},
we cannot generalise this formula to wild invariants by means of a lift.
\end{rem}

\subsubsection{Non-triviality of factor}  We  prove the non-triviality of the factor appearing in Proposition \ref{prop:centre}.
This  uses the non-triviality of $\rhokahn$ for  Platonov's examples (Theorem \ref{thm:nottriv}).  First we recall
some  notions related to tori.  See  \cite{requivtores} as a reference for more details.

Denote for a finite separable field extension $K$ of $k$ by $R_{K/k}(\Gb_m)$ the torus obtained by Weil restriction
of scalars from $K$ to $k$.  Denote furthermore the kernel of the multiplication map 
$R_{K/k}(\Gb_m)\to \Gb_{m,k}$ by $R^1_{K/k}(\Gb_m)$  and the cokernel of the injection $\Gb_{m,k}\to R_{K/k}(\Gb_m)$ by $R_{K/k}(\Gb_m)/\Gb_m$.  Furthermore for any $k$-torus $T$, we denote by $T(k)/R$ the 
$R$-equivalence classes of $T(k)$.  The dual $\hat{T}$ of a $k$-torus $T$ is the character group $\Hom(T,\Gb_m)$.  The dual of $R_{K/k}(\Gb_m)$ is clearly the free abelian group $\Zb[\Gamma]$ for $\Gamma=\Gal(K/k)$. 
The dual of $R^1_{K/k}(\Gb_m)$ is
then $J_\Gamma$, the cokernel of the norm: 
\[\Zb\to \Zb[\Gamma]: a \mapsto \sum_{\gamma_i\in \Gamma} a \gamma_i. \]
The dual of $R_{K/k}(\Gb_m)/\Gb_m$ is the kernel $I_\Gamma$ of the augmentation map:
\[\Zb[\Gamma]\to \Zb:\sum_{\gamma_i\in \Gamma} n_i \gamma_i \mapsto \sum_{\gamma_i \in \Gamma} n_i. \]

Recall that a $k$-torus $F$ is called \textit{flabby} (\textit{flasque}) if $\hat{F}$ is a flabby $\Gamma_k$-module, i.e. $\Ext^1(\hat{F},P)=0$
for any permutation $\Gamma_k$ module $P$  (for equivalent definitions see ibid., Lem. 1).
A flasque resolution of a $k$-torus $T$ is an exact sequence of $k$-tori
\[ 0 \to S \to E \to T \to 0 \]
with $E$ quasi-trivial (i.e. $\hat{E}$ is a permutation module) and $S$ flabby.  This always exists and if $T$ is split
by a field extension $K$, then $E$ and $S$ can also be chosen to be split by $K$.

\begin{thm}
Let $k$ be a $p$-adic field containing a $n^3$-th primitive root unity.  Suppose
$A=(a,t_1)_n\otimes (c,t_2)_n$ is a division $k((t_1))((t_2))$-algebra, then 
\[ \rhokahna{A,k}([\zeta]) = \varphi \left[\lambda\,  h^4_{n}(\{ a,t_1,c,t_2\})\right] \in H^4_{n^2}(k) \]
for $\zeta$ an $n^2$-th primitive root of unity and an integer $\lambda \not\equiv 0 \mod \overline{n^2}$ (and $\varphi$ as
in Proposition \ref{prop:centre}).
A fortiori, $j(0,n)\not\equiv 0 \mod \overline{n^2}$ for any $n$.
\end{thm}

\begin{proof}
We know by Theorem \ref{thm:nottriv} that $\rhokahna{A}:\SKb_1(A)(k) \to H^4_{n^2}(k)$ is not trivial and
moreover $\SKb_1(A)(k)=\Zb/n$ and $H^4_{n^2}(k)\cong \Zb/{n^2}$.  We  prove that the image of
$\mu_{n^2}(k)\cong \Zb/n^2$ inside  $\SKb_1(A)(k)$ is all of $\SKb_1(A)(k)$.  In that case, $\rhokahna{A}([\zeta])$ is not
trivial in $H^4_{n^2}(k)$ (and in $H^4_{\overline{n^2}}(k)\cong \Zb/\overline{n^2}$) so that $j(0,n)\not\equiv 0 \mod \overline{n^2}$.

To prove the statement, let $L=k(\sqrt[n]{a},\sqrt[n]{b})$ and $\Gamma=\Gal(L/k)\cong \Zb/n \times \Zb/n$.  Then
by taking residues on $k((t_1))((t_2))$ with respect to $t_1$ and $t_2$, Platonov proves 
$\SKb_1(A)(k)\cong \hat{H}^{-1}(\Gamma,L^\times)$ where the cohomology group is a Tate cohomology group (see e.g. \cite[Def. 6.2.4]{weibel}) - also use \cite[Thms. 4.17 \& 5.7]{sk1niettriveng} and \cite[(6.15)]{wadsworth}).   On the other hand,
$\hat{H}^{-1}(\Gamma,L^\times)=T(k)/R$ for $T=R^1_{L/k}(\Gb_m)$ \cite[Prop. 15]{requivtores}.  The resulting isomorphism
$\SKb_1(A)(k)\cong T(k)/R$ is a specialisation morphism (in $t_1$ and $t_2$) \cite[(6.9) \& (6.10)]{wadsworth} 
so that the composite $\mu_{n^2}(k)\to \SKb_1(A)(k)\cong T(k)/R$ is the canonical morphism $\mu_{n^2}(k)\to T(k)/R$.
It suffices  to prove that the latter is surjective.

First take a flabby resolution $1\to S \to E \to T \to 1$ of $L$-split tori, 
then $H^1(k,S)=T(k)/R$ (loc. cit., Thm. 2).  The evaluation morphism $S\times \hat{S}\to \Gb_m$
induces a perfect pairing \cite{nakayamadual,tatedual}:
\[ H^1(k,S)\times H^1(k,\hat{S})\to H^2(k,\Gb_m) \cong \Qb/\Zb. \]
Moreover $H^1(k,S)\cong H^1(\Gamma,S(L))$ as this follows from the inflation-restriction exact sequence \cite[3.3.14]{gilleszam}
and $H^1(L,S)=0$.  The pairing above can be modified to a pairing
\[ H^1(\Gamma,S(L))\times H^1(\Gamma,\hat{S}(L))\to \Br(L/k) \cong \Zb/n^2\Zb. \]
Note that 
$\mu_{n^2}(k)\subset T$ so that we get a dual map $\hat{T}\to \Zb/n^2\Zb$.
Using the flabby resolution and the pairing  $T(k)\times \hat{T}(L) \to L^\times$,  
we get the following commutative diagram of pairings:
\[
 \xymatrix@C-25pt{
 H^1(k,S)  & \times & H^1(k,\hat{S}) \ar[d]^{\cong} \ar[rrrrrr] & & & & & & H^2(k,\Gb_m) \cong \Qb/\Zb \\
 H^1(\Gamma,S(L)) \ar[u]^{\cong} & \times & H^1(\Gamma,\hat{S}(L)) \ar[d] \ar[rrrrrr]& & & & & & \Br(L/k) \ar@{^{(}->}[u] \ar@{=}[d] \\
T(k) \ar[u] & \times & H^2(\Gamma,\hat{T}(L)) \ar[d] \ar[rrrrrr]& & & & & & \Br(L/k) \ar@{=}[d] \\
\mu_{n^2}(k) \ar[u] & \times & H^2(\Gamma,\Zb/n^2) \ar[rrrrrr]& & & & & & \Br(L/k).
}
\]
The bottom pairing is perfect as $\mu_{n^2}(k)\cong \Zb/n^2$; 
note that the bottom square comes from the compatibility of the pairings
\[
 \xymatrix@C-25pt{
T(k) & \times & \hat{T}(L) \ar[d] \ar[rrrrrr]& & & & & & L^\times \ar@{=}[d] \\
\mu_{n^2}(k) \ar[u] & \times & \Zb/n^2 \ar[rrrrrr] & & & & & & L^\times.
}
\]
As $H^1(k,S)=T(k)/R\cong \Zb/n$, to prove the surjectivity of $\mu_{n^2}\to T(k)/R$ it suffices
to prove the injectivity of $H^1(k,\hat{S})\to H^2(\Gamma,\Zb/n^2)$.  Since $H^1( \Gamma,\hat{E}(L))=0$,
this comes down to proving the injectivity of $H^2(\Gamma,\hat{T})\to H^2(\Gamma,\Zb/n^2)$.  
This morphism fits into an exact sequence
\[ H^2(\Gamma,I_\Gamma) \to H^2(\Gamma,\hat{T})\to H^2(\Gamma,\Zb/n^2)  \]
because of the exact sequence of group functors
\[ 0 \to \mu_{n^2} \to T \to R_{L/k}(\Gb_m)/\Gb_m \to 0. \]
Clearly $T \to R_{L/k}(\Gb_m)/\Gb_m$ factors through $R_{L/k}(\Gb_m)$, so that 
$H^2(\Gamma,I_\Gamma) \to H^2(\Gamma,\hat{T})$ factors through 
$H^2(\Gamma,\Zb[\Gamma])$ which is trivial by Shapiro's Lemma.  This proves the desired injectivity.
\end{proof}

\begin{rem}
 Note that the proof also defines an invariant of the torus $T$ with values inside $H^4_{n^2}(k)$.
\end{rem}

From this we get 
the following corollary.

\begin{corr}
Let $k$ be a field containing an $l^2$-th primitive root of unity (for $l\neq \car( k)$ any prime) and let $A=(a,b)_l\otimes (c,d)_l$ be a product of two symbol algebras.
If  $\{a,b,c,d\}\neq 0\in K^M_4(k)/l$, then $\SKb_1(A)\neq 0$.
\end{corr}

\begin{proof}
For a field $k$ of characteristic 0, the corollary follows from
the previous theorem.

Let $k$  be a field of $\car(k)=p>0$ and let $l\neq p$ be a prime and assume $k$ to contain an $l^2$-th primitive root $\zeta\in \kbar$.  Take any product of two symbol $k$-algebras $A=(a,b)_l \otimes (c,d)_l$
for $a,b,c,d\in k^\times$. Let $R$ be a $p$-ring with residue field $k$ and fraction field $K$.  Then $A$
lifts to the central simple $K$-algebra $B=(\tilde{a},\tilde{b})_l \otimes (\tilde{c},\tilde{d})_l$ where $\tilde{a},\tilde{b},\tilde{c},\tilde{d}$ are lifts from $a,b,c,d \in R$.  Under the injection
$H^4_{l^2}(k)\to H^4_{l^2}(K)$ induced by \eqref{eq:splitting}, $\varphi \bigl[h^4_{l,k}(\{a,b,c,d\})\bigr]$ is sent
to $\varphi \bigl[ h^4_{l,K}(\{\tilde{a},\tilde{b},\tilde{c},\tilde{d}\})\bigr]$ (with an abuse of notation for $\varphi$ from Proposition \ref{prop:centre}).  This follows from a same splitting for Milnor's K-Theory \cite[Lem. 2.6]{milnor}.

As $\rhokahna{B,K}([\tilde{\zeta}])=\varphi \bigl[j(0,l)\, h^4_{l,K}(\{\tilde{a},\tilde{b},\tilde{c},\tilde{d}\})\bigr]$ (for a lift $\tilde{\zeta}\in \mu_{l^2}(K)$ of $\zeta$), we
find from the construction in Theorem \ref{thm:lift} that
$\rhokahnga{A,k}([\zeta]) = \varphi \bigl[j(0,l)\, h^4_{l,k}(\{a,b,c,d\})\bigr]$.  On the other hand, as Kahn's invariant generates the
invariant group
(\S \ref{sec:defkahn} \ref{sec:cyclic}),  there is an integer $d$ such that
$\rhokahnga{A}= d\, \rhokahna{A}$. From this the result follows.
\end{proof}

By the proof, it even suffices to prove the non-triviality of the symbols  in characteristic zero to obtain the non-triviality of the symbols in moderate positive characteristic.  (Use a $p$-ring and the splitting in Milnor's K-Theory (loc. cit.)).

\fontsize{10}{12} \selectfont
\def\ccprime{$'$}


\addcontentsline{toc}{section}{Bibliography} 
\begin{thebibliography}{KMRT}\setlength{\itemsep}{-5pt}

\bibitem[AG]{ausgold}
Maurice Auslander and Oscar Goldman.
\newblock The {B}rauer group of a commutative ring.
\newblock {\em
  \href{http://www.jstor.org/stable/1993378?origin=crossref}{Trans. Amer. Math.
  Soc.}}, 97:367--409, 1960.

\bibitem[Alb]{albertbiquat}
Adrian Albert.
\newblock Normal division algebras of degree four over an algebraic field.
\newblock {\em
  \href{http://www.jstor.org/stable/1989546?origin=crossref}{Trans. Amer. Math.
  Soc.}}, 34(2):363--372, 1932.

\bibitem[Bae]{baeza}
Ricardo Baeza.
\newblock {\em Quadratic forms over semilocal rings}.
\newblock Lecture Notes in Mathematics, Vol. 655. Springer-Verlag, Berlin,
  1978.

\bibitem[BK]{blochkato}
Spencer Bloch and Kazuya Kato.
\newblock {$p$}-adic \'etale cohomology.
\newblock {\em
  \href{http://www.numdam.org/item?id=PMIHES_1986__63__107_0}{Publ. Math. Inst.
  Hautes {\'E}tudes Sci.}}, (63):107--152, 1986.

\bibitem[Coh]{cohen}
Irvin Cohen.
\newblock On the structure and ideal theory of complete local rings.
\newblock {\em
  \href{http://www.jstor.org/stable/1990313?origin=crossref}{Trans. Amer. Math.
  Soc.}}, 59:54--106, 1946.

\bibitem[CTS]{requivtores}
Jean-Louis Colliot-Th{\'e}l{\`e}ne and Jean-Jacques Sansuc.
\newblock La {$R$}-\'equivalence sur les tores.
\newblock {\em
  \href{http://www.numdam.org/item?id=ASENS_1977_4_10_2_175_0}{Ann. Sci.
  \'Ecole Norm. Sup. (4)}}, 10(2):175--229, 1977.

\bibitem[Dra]{draxl}
Peter Draxl.
\newblock {\em Skew Fields}, volume~81 of {\em London Mathematical Society
  Lecture Note Series}.
\newblock Cambridge University Press, Cambridge, 1983.

\bibitem[Gar]{garibaldi}
Skip Garibaldi.
\newblock Cohomological invariants: exceptional groups and spin groups.
\newblock {\em Mem. Amer. Math. Soc.}, 200(937):xii+81, 2009.
\newblock With an appendix by Detlev W. Hoffmann.

\bibitem[GMS]{cohinv}
Skip Garibaldi, Alexander Merkurjev, and Jean-Pierre Serre.
\newblock {\em {Cohomological invariants in {Galois} cohomology}}, volume~28 of
  {\em University Lecture Series}.
\newblock Amer. Math. Soc., 2003.

\bibitem[Gro1]{ega4}
Alexander Grothendieck.
\newblock {\em {\'El\'ements de G\'eom\'etrie Alg\'ebrique IV, \'Etude locale
  des sch\'emas et des morphismes de sch\'emas, Premi\`ere Partie}}, volume~20
  of {\em
  \href{http://www.numdam.org/numdam-bin/fitem?id=PMIHES_1964__20__5_0}{Publ.
  Math. Inst. Hautes {\'E}tudes Sci.}}
\newblock Bures-sur-Yvette, 1964.

\bibitem[Gro2]{grothbrauer}
Alexander Grothendieck.
\newblock Le groupe de {B}rauer : {I}. {A}lg\`ebres d'{A}zumaya et
  interpr\'etations diverses.
\newblock {\em
  \href{http://www.numdam.org/numdam-bin/fitem?id=SB_1964-1966__9__199_0}{S\'e%
minaire Bourbaki}}, 9:199--219, 1964-1966.
\newblock Expos\'e No. 290.

\bibitem[GS]{gilleszam}
Philippe Gille and Tam{\'a}s Szamuely.
\newblock {\em Central Simple Algebras and Galois Cohomology}, volume 101 of
  {\em Cambridge studies in advanced mathematics}.
\newblock Cambridge University Press, Cambridge, 2006.

\bibitem[Izh]{izboldhin}
Oleg Izhboldin.
\newblock On the cohomology groups of the field of rational functions.
\newblock In {\em Mathematics in {S}t.\ {P}etersburg}, volume 174 of {\em Amer.
  Math. Soc. Transl. Ser. 2}, pages 21--44. Amer. Math. Soc., Providence, RI,
  1996.

\bibitem[Kah1]{kahnquadform}
Bruno Kahn.
\newblock {\em Formes quadratiques sur un corps}, volume~15 of {\em {Cours
  Sp\'ecialis\'es}}.
\newblock {Soci\'et\'e Math\'ematique de France}, 2008.

\bibitem[Kah2]{kahnsk12}
Bruno Kahn.
\newblock Cohomological approaches to {$\SK_1$} and {$\SK_2$} of central simple
  algebras.
\newblock {\em
  \href{http://www.mathematik.uni-bielefeld.de/LAG/man/368.html}{Preprint}},
  2009.

\bibitem[Kat1]{katogalcoh}
Kazuya Kato.
\newblock Galois cohomology of complete discrete valuation fields.
\newblock In {\em Algebraic {K-T}heory}, volume 967 of {\em Lecture notes in
  mathematics}, pages 215--238, Berlin, 1982.

\bibitem[Kat2]{katoquadform}
Kazuya Kato.
\newblock Symmetric bilinear forms, quadratic forms and {M}ilnor {$K$}-theory
  in characteristic two.
\newblock {\em
  \href{http://www.digizeitschriften.de/index.php?id=loader&tx_jkDigiTools_pi1%
[IDDOC]=376270}{Invent. Math.}}, 66(3):493--510, 1982.

\bibitem[KMRT]{kmrt}
Max-Albert Knus, Alexander Merkurjev, Markus Rost, and Jean-Pierre Tignol.
\newblock {\em The book of involutions}, volume~44 of {\em Amer. Math. Soc.
  Colloq. Publ.}
\newblock 1998.

\bibitem[Knu]{Knus}
Max-Albert Knus.
\newblock {\em Quadratic and {H}ermitian forms over rings}, volume 294 of {\em
  Grundlehren der Mathematischen Wissenschaften}.
\newblock Springer-Verlag, Berlin, 1991.

\bibitem[Mat]{matsumura}
Hideyuki Matsumura.
\newblock {\em Commutative ring theory}, volume~8 of {\em Cambridge Studies in
  Advanced Mathematics}.
\newblock Cambridge University Press, Cambridge, 1986.
\newblock Translated from the Japanese by M. Reid.

\bibitem[Mer1]{merkrostinv}
Alexander Merkurjev.
\newblock {$K$}-theory of simple algebras.
\newblock In {\em {$K$}-theory and algebraic geometry: connections with
  quadratic forms and division algebras ({S}anta {B}arbara, {CA}, 1992)},
  volume~58 of {\em Proc. Sympos. Pure Math.}, pages 65--83. Amer. Math. Soc.,
  Providence, RI, 1995.

\bibitem[Mer2]{invalggroup}
Alexander Merkurjev.
\newblock Invariants of algebraic groups.
\newblock {\em
  \href{http://www.reference-global.com/doi/abs/10.1515/crll.1999.508.127}{J.
  reine angew. Math.}}, 508:127--156, 1999.

\bibitem[Mer3]{mersuslinbiquat}
Alexander Merkurjev.
\newblock The group {$SK\sb 1$} for simple algebras.
\newblock {\em
  \href{http://www.ams.org/leavingmsn?url=http://dx.doi.org/10.1007/s10977-006%
-0021-4}{$K$-Theory}}, 37(3):311--319, 2006.

\bibitem[Mil]{milnor}
John Milnor.
\newblock Algebraic {$K$}-theory and quadratic forms.
\newblock {\em
  \href{http://www.digizeitschriften.de/index.php?id=loader&tx_jkDigiTools_pi1%
[IDDOC]=374693}{Invent. Math.}}, 9:318--344, 1969/1970.

\bibitem[Nak]{nakayamadual}
Tadasi Nakayama.
\newblock Cohomology of class field theory and tensor product modules. {I}.
\newblock {\em \href{http://dx.doi.org/10.2307/1969962}{Ann. of Math. (2)}},
  65:255--267, 1957.

\bibitem[NM]{nakmat}
Tadasi Nakayama and Yoz{\^o} Matsushima.
\newblock \"{U}ber die multiplikative {G}ruppe einer {$p$}-adischen
  {D}ivisionsalgebra.
\newblock {\em
  \href{http://projecteuclid.org/DPubS?service=UI&version=1.0&verb=Display&han%
dle=euclid.pja/1195573246}{Proc. Imp. Acad. Tokyo}}, 19:622--628, 1943.

\bibitem[OVV]{orlvisvoe}
Dmitri Orlov, Alexander Vishik, and Vladimir Voevodsky.
\newblock An exact sequence for {$K\sp M\sb \ast/2$} with applications to
  quadratic forms.
\newblock {\em \href{http://annals.math.princeton.edu+p01.xhtml}{Ann. of
  Math.}}, 165(1):1--13, 2007.

\bibitem[Pla]{sk1niettriveng}
Vladimir Platonov.
\newblock The {Tannaka-Artin problem and reduced K-}theory.
\newblock {\em Math. USSR Izv.}, 10(2):211--243, 1976.
\newblock English translation.

\bibitem[Ros1]{rostmodcyc}
Markus Rost.
\newblock Chow {Groups with C}oefficients.
\newblock {\em \href{http://www.math.uiuc.edu/documenta/vol-01/16.html}{Doc.
  Math. J. DMV}}, 1:319--393, 1996.

\bibitem[Ros2]{rostblken}
Markus Rost.
\newblock The basic correspondence of a splitting variety.
\newblock 1998.
\newblock \href{http://www.math.uni-bielefeld.de/~rost/basic-corr.html}{Notes
  downloadable from his website}.

\bibitem[RTY]{rehmanea}
Ulf Rehmann, Sergey Tikhonov, and Vyacheslav Yanchevski\u{\i}.
\newblock Symbols and cyclicity of algebras after a scalar extension.
\newblock {\em Fundam. Prikl. Mat.}, 14(6):193--209, 2008.

\bibitem[Ser]{serregalcoh}
Jean-Pierre Serre.
\newblock {\em Galois {Cohomology}}.
\newblock Springer Monographs in Mathematics. Springer-Verlag, Berlin, 2002.

\bibitem[SGA]{sga33}
{\em {S}ch\'emas en groupes. {III}: {S}tructure des sch\'emas en groupes
  r\'eductifs}.
\newblock S\'eminaire de G\'eom\'etrie Alg\'ebrique du Bois Marie 1962/64 (SGA
  3). Dirig\'e par M. Demazure et A. Grothendieck. Lecture Notes in
  Mathematics, Vol. 153. Springer-Verlag, Berlin, 1962/1964.

\bibitem[Sus1]{suslinconj}
Andrei Suslin.
\newblock {$SK\sb 1$} of division algebras and {G}alois cohomology.
\newblock In {\em Algebraic {$K$}-theory}, volume~4 of {\em Adv. Soviet Math.},
  pages 75--99. Amer. Math. Soc., Providence, RI, 1991.

\bibitem[Sus2]{suslin}
Andrei Suslin.
\newblock {$SK_1$} of division algebras and {G}alois cohomology revisited.
\newblock In {\em Proceedings of the {S}t. {P}etersburg {M}athematical
  {S}ociety. {V}ol. {XII}}, volume 219 of {\em Amer. Math. Soc. Transl. Ser.
  2}, pages 125--147, Providence, RI, 2006. Amer. Math. Soc.

\bibitem[Tat]{tatedual}
John Tate.
\newblock The cohomology groups of tori in finite {G}alois extensions of number
  fields.
\newblock {\em
  \href{http://projecteuclid.org/getRecord?id=euclid.nmj/1118801784}{Nagoya
  Math. J.}}, 27:709--719, 1966.

\bibitem[Voe]{voevodblken}
Vladimir Voevodsky.
\newblock On {Motivic C}ohomology with $\mathbb{Z}/l$ coefficients.
\newblock {\em \href{http://www.math.uiuc.edu/K-theory/0639/}{Preprint}}, 2009.

\bibitem[Wad]{wadsworth}
Adrian Wadsworth.
\newblock Valuation theory on finite dimensional division algebras.
\newblock In {\em Valuation theory and its applications, {V}ol. {I}
  ({S}askatoon, {SK}, 1999)}, volume~32 of {\em Fields Inst. Commun.}, pages
  385--449. Amer. Math. Soc., Providence, RI, 2002.

\bibitem[Wan]{wang}
Shianghaw Wang.
\newblock On the commutator group of a simple algebra.
\newblock {\em \href{http://www.jstor.org/stable/2372036?origin=crossref}{Amer.
  J. Math.}}, 72:323--334, 1950.

\bibitem[Wei1]{weibel}
Charles Weibel.
\newblock {\em An introduction to homological algebra}, volume~38 of {\em
  Cambridge Studies in Advanced Mathematics}.
\newblock Cambridge University Press, Cambridge, 1997.

\bibitem[Wei2]{weibelblk}
Charles Weibel.
\newblock The norm residue isomorphism theorem.
\newblock {\em \href{http://dx.doi.org/10.1112/jtopol/jtp013}{J. Topol.}},
  2(2):346--372, 2009.

\bibitem[Wit]{witt}
Ernst Witt.
\newblock Zyklische {K\"orper und Algebren der Charakteristic $p$ vom Grad}
  $p^n$.
\newblock {\em
  \href{http://www.digizeitschriften.de/resolveppn/GDZPPN002173824}{J. reine
  angew. Math.}}, 176:126--140, 1937.

\bibitem[Wou]{wsuspos}
Tim Wouters.
\newblock L'invariant de {S}uslin en caract\'eristique positive.
\newblock {\em Preprint}, 2009.

\end{thebibliography}
\end{document}